\documentclass[12pt]{amsart}
\usepackage{amssymb,amsmath,amsthm,fullpage}

\newtheorem{thm}{Theorem}[section]
\newtheorem{prop}[thm]{Proposition}
\newtheorem{lem}[thm]{Lemma}
\newtheorem{cor}[thm]{Corollary}

\theoremstyle{definition}
\newtheorem{defn}[thm]{Definition}

\theoremstyle{remark}
\newtheorem{rem}[thm]{Remark}
\newtheorem{remark}[thm]{Remark}

\newcommand{\fjkt}{\frac{j+k}2}
\newcommand{\ajkz}{\A{jk}{z}}
\newcommand{\ajkw}{\A{jk}{w}}

\newcommand{\ajko}{\A{jk}{0}}
\newcommand{\ajkoo}{\A{j(k+1)}{0}}

\newcommand{\sumjko}{\sum_{\atopp{j\geq 1}{k\geq 0}}}
\renewcommand{\H}{\mathcal{H}}

\newcommand{\R}{\mathbb R}
\newcommand{\Z}{\mathbb Z}

\newcommand{\N}{\mathbb N}
\newcommand{\C}{\mathbb C}

\DeclareMathOperator{\supp}{supp}
\DeclareMathOperator{\Imm}{Im}

\newcommand{\p}{\partial}

\newcommand{\ges}{\gtrsim}
\newcommand{\z}{\bar z}
\newcommand{\w}{\bar w}
\newcommand{\dbar}{\bar\partial}

\newcommand{\dbarb}{\bar\partial_b}
\newcommand{\dbarbs}{\bar\partial_b\hspace{-3.5pt}{}^*}
\newcommand{\Boxb}{\Box_b}
\newcommand{\vp}{\varphi}

\newcommand{\atopp}[2]{\genfrac{}{}{0pt}{2}{#1}{#2}}

\newcommand{\A}[2]{a_{#1}^{#2}}
\newcommand{\sjk}{\sum_{j,k\geq 1}}
\newcommand{\sab}{\sum_{\alpha,\beta\geq 1}}
\newcommand{\nn}{\nonumber}

\newcommand{\LL}{\bar L}

\newcommand{\Zstp}{Z_{\tau p}}
\newcommand{\Zbstp}{\overline{Z}_{\tau p}}
\newcommand{\Zstpw}{Z_{\tau p,w}}
\newcommand{\Zbstpw}{\overline{Z}_{\tau p,w}}
\newcommand{\Zstpz}{Z_{\tau p,z}}
\newcommand{\Zbstpz}{\overline{Z}_{\tau p,z}}
\newcommand{\Zstpxi}{Z_{\tau p,\xi}}

\newcommand{\Mtp}{M_{\tau p}}
\newcommand{\Mtpzw}{M_{\tau p}^{z,w}}

\newcommand{\Wstpw}{W_{\tau p,w}}
\newcommand{\Wbstpw}{\overline{W}_{\tau p,w}}

\newcommand{\epl}{e^{i \tau T(w,z)}}
\newcommand{\emi}{e^{-i \tau T(w,z)}}
\newcommand{\T}{\epl \frac{\p}{\p\tau} \emi}

\newcommand{\ep}{\epsilon}

\newcommand{\Boxtp}{\Box_{\tau p}}
\newcommand{\Boxwtp}{\Boxw_{\tau p}}
\newcommand{\Boxw}{{\widetilde \Box}}
\newcommand{\Boxtpz}{\Box_{\tau p,z}}
\newcommand{\Boxtpxi}{\Box_{\tau p,\xi}}

\newcommand{\I}{\mathcal{I}}

\newcommand{\Boxwtpz}{\Boxw_{\tau p,z}}

\renewcommand{\H}{\mathcal{H}_{\tau p}}

\newcommand{\Htp}{H_{\tau p}}
\newcommand{\Hwtp}{{\tilde H}_{\tau p}}

\newcommand{\la}{\langle}
\newcommand{\ra}{\rangle}

\begin{document}

\begin{abstract} In this article, we establish Gaussian 
decay for the $\Boxb$-heat kernel on polynomial models in $\C^2$. Our
technique attains the exponential decay 
via a partial Fourier transform. On the transform side, the problem
becomes finding  quantitative smoothness estimates on a 
heat kernel associated to the weighted
$\dbar$-operator on $L^2(\C)$. The bounds are established 
with Duhamel's formula and careful estimation.
\end{abstract}

\title {Heat Kernels, Smoothness Estimates and Exponential Decay}

\author{Albert Boggess and Andrew Raich}

\thanks{The second author  is partially supported by NSF grant DMS-0855822.}

\address{
Department of Mathematics\\ Texas A\&M University\\ Mailstop 3368  \\ College Station, TX
77845-3368 }
\address{
Department of Mathematical Sciences \\ 1 University of Arkansas \\ SCEN 327 \\ Fayetteville, AR 72701}


\email{boggess@math.tamu.edu, araich@uark.edu}

\maketitle

%
%
\section{Introduction}\label{sec:intro}
The purpose of this article is to prove Gaussian decay for the $\Box_b$ heat kernel on 
polynomial models in $\C^2$ and introduce a class of estimates called quantitative smoothness estimates.
We develop a new method for obtaining exponential decay via the Fourier transform as our newly developed quantitative
smoothness estimates characterize such functions. 
We are then able to show that the kernel associated to a weighted $\dbar$-operator on $\C$ satisfies a number of
quantitative smoothness estimates, and this allows us to recover the Gaussian decay estimate for the $\Box_b$ heat kernel.

\subsection{Polynomial models in $\C^2$}
\begin{defn}\label{defn:polynomial model}A \textbf{polynomial model} $M\subset\C^2$ is a manifold of the form
\[
M = \{(z,w)\in\C^2 : \Imm w = p(z)\}
\]
where $p:\C\to\R$ is a subharmonic, nonharmonic polynomial.
\end{defn}
$M$ is the boundary of an unbounded pseudoconvex domain called a polynomial model domain. For example, if $p(z) = |z|^2$, then $M$
is the Heisenberg group $\mathbb{H}^1$ and is the boundary of the Siegel upper-half space.
$M \cong \C\times\R$ with the identification $(z,t+ip(z)) \longleftrightarrow (z,t)$. We will not distinguish $M\subset\C^2$ with
its image $\C\times\R$. The tangential Cauchy-Riemann operator $\dbar_b$ on $M$ can be identified with the vector field
\[
\LL = \frac{\p}{\p\z} - i\frac{\p p}{\p\z}\frac{\p}{\p t}
\]
and $\dbarbs$ on $M$ can be identified with the vector field
\[
L = \frac{\p}{\p z} + i \frac{\p p}{\p z}\frac{\p}{\p t}.
\]
The Kohn Laplacian $\Box_b = \dbarb\dbarb^* + \dbarb^*\dbarb$ is then
identified with $\Boxb = -\LL L$ on $(0,1)$-forms and $\Boxb = -L\LL$ on functions.

The $\dbarb$-complex on unbounded CR-manifolds is a relatively unexplored subject, and polynomial models provide a model case
to study. In addition, polynomial models provide a good local approximation to a CR manifold of finite type and have
been used to prove local results in dimension 3, see e.g., \cite{Chr89e}.
An advantage of working with polynomial models is that the nonisotropic control metric is globally defined \cite{NaStWa85}. This
is one reason that, with notable exceptions such as 
Kang's work closed range of $\dbarb$ on weighted $L^2$ when $p(z)$ is radial \cite{Kan89}, 
a major focus of the analysis is establishing pointwise estimates
on integral kernels (in terms of the control metric) \cite{NaRoStWa89,NaSt01h,NaSt06,Rai06f}.
As mentioned above, the prototype polynomial model is the Heisenberg group. Analysis on it, however, is aided by the fact that it is a Lie group, whereas
the generic polynomial model lacks any group structure. 

\subsection{$\Boxb$-heat kernel}
Our goal is the prove pointwise estimates on the $\Box_b$-heat kernel and its derivatives. For $\alpha=(z,t_1), q=(w,t_2)\in\C\times\R$,
The $\Boxb$-heat equation
is
\begin{equation}\label{eqn:Boxb heat equation}
\begin{cases}
\displaystyle \frac{\p u}{\p s} + \Boxb u =0  &\text{in } (0,\infty)\times\C\times\R\\
u(0,\alpha) = f(\alpha) &\text{on } \{s=0\}\times\C\times\R
\end{cases}
\end{equation}
As in \cite{NaSt01h}, we solve (\ref{eqn:Boxb heat equation}) using the heat semigroup $e^{-s\Boxb}$. 
In particular, there exists 
the heat kernel $\H(s,\alpha,\beta)$ that is $C^\infty$ off of the diagonal 
$\{s=0,\ \alpha=\beta\}$ and if $u(s,\alpha) = (e^{-s\Boxb}f)(\alpha)$, then
\[
u(s,\alpha) = \int_{\C\times\R} \H(s,\alpha,\beta) f(\beta)\, d\beta
\]
and $u$ solves \eqref{eqn:Boxb heat equation}. 

Solving the $\Boxb$-heat equation has many applications to the theory of $\Boxb$. In particular, the spectral theorem
for unbounded operators allows us to recover  the 
Szeg\"o kernel $S$ as $S = \lim_{s\to\infty}e^{-s\Boxb}$ and the relative fundamental solution which
is given by $\int_0^\infty e^{-s\Boxb}(I-S)\, ds$. 
Moreover, one method to bound the number of eigenvalues below a fixed threshold requires estimates on the trace
(in the operator sense) of the heat kernel for small time.

In \cite{NaSt01h}, Nagel and Stein prove that the heat kernel $\H(s,\alpha,\beta)$ satisfies rapid decay, and our goal is to present
a calculation to improve the decay to exponential decay. Similar results have been obtained in an unpublished
result by Nagel and M\"uller and independently by Street \cite{Str09h}. Nagel and M\"uller adapt the technique of
\cite{JeSa86} while Street adapts the technique of \cite{Sik04, Rai07}. The disadvantage of the techniques of Nagel/M\"uller and
Street is that they do not seem to generalize to higher dimensions. Our ideas ought to generalize, and we plan to pursue this
in a future work.

\subsection{Weighted operators on $L^2(\C)$.}
Since the operator $\LL$ is translation invariant in $t$, we can study $\Boxb$ by taking a partial Fourier transform in $t$. 
Studying $\dbarb$ on polynomial models via the partial Fourier transforms has been a fruitful method
\cite{Na86,Christ91,Has94,Has95,Has98,Rai06h,Rai06f,Rai07,BoRa09,Rai10h,BoRa10}

If $f(z,t)$ is a function
on $\C\times\R$, we define the partial Fourier transform of $f$ by
\[
\hat f(z,\tau) = \int_{\R} e^{-it\tau} f(z,t)\, dt.
\]
Under the partial Fourier transform
\[
\LL \mapsto \Zbstp = \frac{\p}{\p\z} + \tau\frac{\p p}{\p\z} = e^{-\tau p}\frac{\p}{\p\z} e^{\tau p}
\qquad\text{and}\qquad
L \mapsto \Zstp = \frac{\p}{\p z} - \tau \frac{\p p}{\p z}  = e^{\tau p}\frac{\p}{\p z} e^{-\tau p}.
\]
Similarly, the Kohn Laplacian
$\Box_b$ on $(0,1)$-forms maps to $\Boxtp = -\Zbstp\Zstp$ and $\Box_b$ 
on functions maps to $\Boxwtp= -\Zstp\Zbstp$.
We will see below that understanding the $\tau$-derivative of the $\Boxb$-heat kernel 
$\tau$ is essential for proving its exponential decay estimates. 

Applying the partial Fourier transform to (\ref{eqn:Boxb heat equation}), we have the heat equations
\begin{equation}\label{eqn:Boxtp heat equation}
\begin{cases} {\displaystyle \frac{\p u}{\p s}} + \Boxtpz u=0 &\text{on }(0,\infty)\times\C \vspace*{.1in}\\
u(0,z) = f(z) & \text{in } \{s=0\}\times\C \end{cases}
\end{equation}
and
\begin{equation}\label{eqn:Boxwtp heat equation}
\begin{cases} {\displaystyle \frac{\p \tilde u}{\p s}} + \Boxwtpz \tilde u=0    &\text{on }(0,\infty)\times\C \vspace*{.1in}\\
\tilde u(0,z) = \tilde f(z)   &\text{in } \{s=0\}\times\C.\end{cases}
\end{equation}
Note that $u$ and $\tilde u$ are no longer functions of $t$ as they were in (\ref{eqn:Boxb heat equation}), and generically, $u$ and 
$\tilde u$ are not functions of $\tau$ as we think of $\tau$ as a parameter.
Let $\Htp(s,z,w)$ and $\Hwtp(s,z,w)$ be the heat kernels associated to \eqref{eqn:Boxtp heat equation} and \eqref{eqn:Boxwtp heat equation}, respectively.
It turns out that these heat equations are dual to one another in the following sense:
if 
\[
\Boxwtpz = -\Zstpz\Zbstpz,
\]
then 
\[
\Box_{-\tau p,z} = \overline{\Boxwtpz}.
\]
This equality, coupled with the fact that $\Boxwtpz$ is self-adjoint, forces
\begin{equation}
\label{eqn:HtwiddleH}
H_{-\tau  p}(s,z,w) = \overline{\Hwtp(s,z,w)} = \Hwtp(s,w,z).
\end{equation}
In other words, the roles and $\Zbstp$ and $\Zstp$ switch when $\tau<0$. This is a key equality for handling both the $\tau<0$ case as well as the case when
$\Box_b = -L \LL$. See Remark \ref{rem:main theorem reductions} for details.

\subsection{Outline of the article.}
In Section \ref{sec:results}, we introduce notation and formulate the Gaussian decay result on polynomial models, Theorem \ref{thm:Boxb heat kernel decay}. 
Generically, the exponential  decay of the $\Boxb$-heat kernel is of the form $e^{-a |t|^{1/\beta}}$ where
$\beta\geq 1$. Since we are using the Fourier transform to recover the estimates, we need to find a condition that
is tractable across the transform. To do this, we characterize $e^{-a |t|^{1/\beta}}$ in terms of
$\| |t|^\ell e^{-a |t|^{1/\beta}}\|_{L^\infty(\R)}$ for $\ell\geq 0$ in the spirit of \cite{GeSh67}. This leads to estimates on the 
Fourier transform side that we call quantitative smoothness estimates. This is the content of Section \ref{sec:QSE}.
In Section \ref{sec:heat kernel in terms of QSE}, we recast the Gaussian decay in terms of the quantitative
smoothness estimates. In Section \ref{sec:tau deriv estimates}, we formulate the main result on the quantitative smoothness estimate
of the $\Boxtp$-heat kernel,
Theorem \ref{thm:Htp satisfies QSE}, and show that this result implies Theorem \ref{thm:Boxb heat kernel decay}. 
To establish the estimates of Theorem \ref{thm:Htp satisfies QSE} , we combine Duhamel's principle and a recursion to find a formula for
the $\tau$-derivatives of the $\Boxtp$-heat kernel. This is the content of  Section \ref{sec:a good formula for MH}. 
In Section \ref{sec: proof of QSE thm}, we prove Theorem \ref{thm:Htp satisfies QSE}.

%
%
\section{Results}\label{sec:results}

\subsection{The control metric on $M$}
In \cite{NaStWa85}, Nagel et.\ al.\ prove the existence of the control metric on manifolds such as $M$. We need
to introduce some quantities to write down an equivalent size to the metric (see \cite{NaStWa85,NaSt01h,NaSt01f} for details).
Let
\[
T(w,z) = -2\Imm\Big(\sum_{j\geq 1} \frac{1}{j!} \frac{\p^j p(z)}{\p z^j}(w-z)^j\Big).
\]
and
\begin{align*}
\ajkz &= \frac{1}{j!k!}\frac{\p^{j+k} p(z)}{\p z^j \p\z^k}, &
\Lambda(z,\delta) &= \sjk |\ajkz| \delta^{j+k}, &
\mu(z,\delta) &= \inf_{j,k\geq 1} \Big| \frac{\delta}{\ajkz}\Big|^{\frac 1{j+k}}.
\end{align*}
The functions $\mu$ and $\Lambda$ are relative inverses in the sense that
\[
\mu\big(z,\Lambda(z,\delta)\big) \sim \Lambda\big(\mu(z,\delta)\big) \sim \delta.
\]
We say that $A\sim B$ if there exists a global constant $c$ so that $\frac 1c A \leq B \leq cA$. 
For points $\alpha=(z,t_1)$ and $\beta=(w,t_2)$, the control metric on $M$ is equivalent to (with an abuse of notation)
\[
d(\alpha,\beta) = d(z,w,t_1-t_2) = |z-w| + \mu\big(z,t_1-t_2 + T(z,w)\big),
\]
and with this distance, the volume of a ball of radius $\delta$, $B_d(\alpha,\delta)$ is 
\[
|B_d(\alpha,\delta)| \sim \delta^2\Lambda(z,\delta).
\]
Since $|B_d(\alpha,\delta)|$ does not depend on $t$, we sometimes engage in a small abuse of notation and write $|B_d(z,\delta)|$ in lieu of
$|B_d(\alpha,\delta)|$.
Given points $\alpha,\beta\in \C\times\R$ as above, the volume of the ball centered at $p$ of radius $d(\alpha,\beta)$ is denoted 
$V(\alpha,\beta) = V(z,w,t_1-t_2)$ and 
\begin{multline*}
V(z,w,t_1-t_2) = d(z,w,t_1-t_2)^2\Lambda\big(z,d(z,w,t_1-t_2)\big) \\
\sim \max\Big\{ |z-w|^2\Lambda(z,|w-z|), \mu\big(z,t_1-t_2 + T(w,z)\big)^2 \big|t_1-t_2+T(w,z)\big| \Big\}.
\end{multline*}

As a consequence of the ``twist", $T(w,z)$, the derivative in $\tau$ is the twisted derivative
\[
\Mtpzw = \T = \frac{\p}{\p \tau} - i T(w,z),
\]
as $-i(t+T(w,z)) \vp\ \hat \mapsto\ \Mtpzw\hat\vp(\tau)$.

\subsection{Results}
For the remainder of the paper, consider $z$ and $w$ as fixed points in $\C$. 
Let $J = (j_0,\dots, j_k) \in \{0,1\}^k$ be a multiindex. 
We set $X^J = X^{j_0}\cdots X^{j_k}$ where $X^0 = L$ and $X^1 = \LL$.
We now present our main theorem on polynomial models.
\begin{thm} \label{thm:Boxb heat kernel decay}
Let $\H(s,\alpha,\beta)$ be the $\Box_b$-heat kernel associated to (\ref{eqn:Boxb heat equation}).
Let $J$ and $J'$ be multiindices. There exists positive constants $C,c>0$  so that
\begin{equation}\label{eqn:Boxb heat kernel decay}
|X^J_\alpha X^{J'}_\beta \H(s,\alpha,\beta)| 
\leq C \frac{e^{-c \frac{d(\alpha,\beta)^2}s}} {d(\alpha,\beta)^{|J|+|J'|}V(\alpha,\beta)}
\end{equation}
for all $\alpha$ and $\beta \in M$ and $s>0$.
If $X^J_\alpha X^{J'}_\beta \frac{\p^j}{\p s^j} S(\alpha,\beta)=0$ where $S(\alpha,\beta)$ is the Szeg\"o kernel, then
\begin{equation}\label{eqn:time derivative Boxb heat kernel decay}
\bigg|X^J_\alpha X^{J'}_\beta \frac{\p^j}{\p s^j}\H(s,\alpha,\beta)\bigg| 
\leq C \frac{e^{-c \frac{d(\alpha,\beta)^2}s}} {s^{j+ \frac12(|J|+|J'|)} |B_d(\alpha,\sqrt s)|}
\end{equation}
for all $\alpha$ and $\beta \in M$ and $s>0$.
\end{thm}

\begin{rem}\label{rem:main theorem reductions} There are several reductions that we can make.
\begin{enumerate}\renewcommand{\labelenumi}{(\roman{enumi})}
\item The bounds for $|X^J_\alpha X^{J'}_\beta\H(s,\alpha,\beta)|$ when $d(\alpha,\beta) \sim |z-w|$ are proved in \cite{Rai10h}.
\item Notice that if $s \geq d(\alpha,\beta)^2$, then $\exp(-c_0 d(\alpha,\beta)^2/s) \sim 1$ and provides no
decay as $s\to\infty$. Consequently, the bounds when
$s \geq d(\alpha,\beta)^2$ are a consequence of \cite{NaSt01h} or \cite{Rai10h}.
\item  The estimate (\ref{eqn:time derivative Boxb heat kernel decay}) is only better than 
(\ref{eqn:Boxb heat kernel decay})  for large $s$, i.e., when $s\geq d(\alpha,\beta)^2$. In this 
case, there is decay in $s$ that is simply not present if $X^J_\alpha X^{J'}_\beta \frac{\p^j}{\p s^j} S(\alpha,\beta)\neq0$ as $\lim_{s\to\infty} e^{-s\Box_b}=S$. 
When $X^J_\alpha X^{J'}_\beta \frac{\p^j}{\p s^j} S(\alpha,\beta)=0$, the decay in (\ref{eqn:time derivative Boxb heat kernel decay}) occurs 
because the the derivative of the kernel of heat semigroup $e^{-s\Box_b}$  will coincide with the derivative of the kernel of
the semigroup $e^{-s\Box_b}(I-S)$. The estimates in \eqref{eqn:time derivative Boxb heat kernel decay} follow immediately from 
\eqref{eqn:Boxb heat kernel decay} and the estimates for  the kernel of $e^{-s\Box_b}(I-S)$ proven in
 \cite{NaSt01h} (and they can also be obtained from \cite{Rai10h}). 
Since the constant $c$ is not sharp, the small time estimate in  \eqref{eqn:time derivative Boxb heat kernel decay} is equivalent to the small
time estimate in \eqref{eqn:Boxb heat kernel decay} (with a slight decrease in $c$).
\item Theorem \ref{thm:Boxb heat kernel decay} makes no distinction between $\Box_b = -L \LL$ on functions and $\Box_b = - \LL L$ on
$(0,1)$-forms. $- L \LL\ \hat\mapsto\ \Boxwtp$ while $-\LL L\ \hat\mapsto\ \Boxtp$. Because of \eqref{eqn:HtwiddleH}, understanding
the $\tau>0$ cases for $\Htp(s,z,w)$ and $\Hwtp(s,z,w)$ is equivalent to understanding the $\tau<0$ cases for $\Hwtp(s,z,w)$ and $\Htp(s,z,w)$, respectively.
\item Consequently, it suffices to prove Theorem \ref{thm:Boxb heat kernel decay} when $\Boxb = - \LL L$, $s\leq d(\alpha,\beta)^2$, and
$d(\alpha,\beta) = \mu(z,t_1-t_2+T(w,z))$.
\end{enumerate}
\end{rem}

Thus, the content of Theorem \ref{thm:Boxb heat kernel decay} is to achieve \eqref{eqn:Boxb heat kernel decay} 
for $d(\alpha,\beta) \sim \mu(z,t_1-t_2+T(w,z))$ and $s \leq d(\alpha,\beta)^2$.

%
%
\section{Quantitative smoothness estimates}\label{sec:QSE}
The proof of Theorem \ref{thm:Boxb heat kernel decay} uses the heuristic that decay on the function side corresponds to 
smoothness on the transform side. In particular, we need to understand to the Fourier transforms of 
functions that decay like $e^{-a |t|^{1/\beta}}$ when
$\beta \geq 1$. To do this, we introduce quantitative smoothness estimates.
\begin{defn} \label{defn: L^p QSE}
A smooth function $g:\R\to\C$ satisfies an \textbf{$L^q$-quantitative smoothness estimate of order $\beta$}, abbreviated
$L^q$-QSE,  if
there exist constants $A, C>0$ so that for all integers $\ell\geq 0$,
\[
\| g^{(\ell)} \|_{L^q(\R)} \leq C A^\ell \ell^{\ell\beta}.
\]
\end{defn}
Here, $ g^{(\ell)}$ stands for the derivative of order $\ell$ of $g$.
If $\beta \leq1$ and $g$ satisfies an $L^\infty$-QSE, then $g$ will be in some quasianalytic class 
and extend holomorphically to a strip ($\beta=1$) or to an entire function $(\beta<1)$.
For $\beta>1$, the case of interest here, such functions do not lie in any quasianalytic class. 
This is an immediate consequence of the Denjoy-Carleman Theorem
(see \cite{Rud87}, Theorem 19.11).

\subsection{Explanation of QSE}
The ideas for these calculations can be found in \cite{GeSh67}. 
From first year calculus, we know that 
\begin{equation}\label{eqn:polynomial times exponential decay}
|t|^\gamma e^{-a|t|^{1/\beta}} \leq \big(\frac{\gamma\beta}{ae}\big)^{\gamma\beta}.
\end{equation}
The surprising fact is that this inequality, if it is true for all $\gamma$, is actually equivalent to exponential decay.
We have the following proposition from \cite{GeSh67}.
\begin{prop}\label{prop:exponential decay equivalence}
Let $a,\beta>0$. Then
\begin{equation}
\label{eqn:propdisplay}
e^{- a |t|^{1/\beta}} = \inf_{\gamma\geq 0} \Big( \frac{\gamma\beta}{ae|t|^{1/\beta}}\Big)^{\gamma\beta}
\leq \inf_{\atopp{n\geq 1}{n\in\Z}}\left\{ 1, \  \Big( \frac{n\beta}{ae|t|^{1/\beta}}\Big)^{n\beta} \right\}
\leq e^{e\beta/2} e^{- a |t|^{1/\beta}} 
\end{equation}
\end{prop}
\begin{proof}
We may assume that $t>0$. Let $\nu_\beta(\xi) = \inf_{\gamma\geq 0} \frac{\gamma^{\gamma\beta}}{\xi^\gamma}$ and
$A = (\frac{\beta}{ae})^\beta$.
Note that $\nu_\beta (t/A) $ is the second term in (\ref{eqn:propdisplay}).
We have already seen that  $e^{-a t^{1/\beta}} \leq \nu_\beta (t/A)$. Fix $\xi>0$.
Let $f(\gamma) = \frac{\gamma^{\gamma\beta}}{\xi^\gamma}$. Then
\[
\log f(\gamma) = \gamma\beta\log\gamma - \gamma\log\xi,
\]
so 
\[
\big( \log f(\gamma)\big)' = \frac{f'(\gamma)}{f(\gamma)} = \beta + \beta\log\gamma - \log\xi
\]
and $(\log f(\gamma))'' = \beta/\gamma>0$. Thus, the zero of $(\log f(\gamma))'$ corresponds to a minimum of $f(\gamma)$.
Also, $(\log f(\gamma_0))'=0$ means that $ \gamma_0 = \frac 1e \xi^{1/\beta}$, so 
\[
\min_{\gamma \geq 0} \log f(\gamma) = -\frac 1e \beta \xi^{1/\beta}
\]
and
\[
\min_{\gamma \geq 0}  f(\gamma) = e^{-\frac \beta e |\xi|^{1/\beta}}.
\]
Consequently, we see that $\nu_\beta(t/A) = e^{-a t^{1/\beta}}$ which establishes
the first equality in (\ref{eqn:propdisplay}). The first inequality in (\ref{eqn:propdisplay})
is obvious, so it remains to show the second inequality.

Let $n_0 = \lceil\gamma_0 \rceil$, i.e., the next largest integer greater than $\gamma_0$. By Taylor's Theorem, there
exists $\gamma_1$ so that $\gamma_0 \leq \gamma_1 \leq n_0$ so that
\[
\log f(n_0) = \log f(\gamma_0) + \frac{\beta}{2\gamma_1} (n_0-\gamma_0)^2 \leq
\log f(\gamma_0) + \frac{\beta}{2\gamma_0}.
\]
Thus,
\[
f(n_0) \leq e^{-\frac \beta e |\xi|^{1/\beta}} e^{\frac{\beta e}2 |\xi|^{-1/\beta}}.
\]
If $|\xi|=|t|/A \geq 1$, then $e^{\frac{\beta e}2 |\xi|^{-1/\beta}} \leq e^{\frac{\beta e}2}$
and this establishes the second inequality in (\ref{eqn:propdisplay}) in this case.
On the other hand, if $|\xi|=|t|/A \leq 1$, then note that 
\[
\Big( \frac{n\beta}{ae|t|^{1/\beta}}\Big)^{\beta n}= \frac{n^{n\beta}}{(t/A)^n} >1 \ \ 
\textrm{for} \ n \geq 1.
\]
So the left side of the second inequality in (\ref{eqn:propdisplay}) is 1. On the other hand,
if $t \leq A$, then it is easy to show that the right side of this inequality is 
greater than 1. This concludes the proof of the second inequality and the proof of the proposition
is complete.

\end{proof}

\begin{cor}\label{cor:exponential decay equivalence with estimates}
Let $\beta, A,C>0$ and $\vp:\R\to \C$ be a function that satisfies 
\[
\big\|t^n \vp \big\|_{L^\infty(\R)} \leq C \Big(\frac{n \beta}{ae}\Big)^{n \beta}
\] 
for all integers $n\geq 0$. Then
\[
|\vp(t)| \leq C e^{-a |t|^{1/\beta}}.
\]
\end{cor}
Corollary \ref{cor:exponential decay equivalence with estimates} allows us to connect functions with exponential decay and functions
that satisfy quantitative smoothness estimates. In particular, we have our main result for quantitative smoothness estimates.

\begin{thm}  \label{thm:QSE estimates}
Let $\vp:\R\to \C$.
\begin{enumerate}
\item Suppose there exist constants $a, \beta >0$ so that $|\vp(t)| \leq C e^{-a|t|^{1/\beta}}$. If $A = (\frac{\beta}{ae})^\beta$, then it follows that 
\[
|t^n \vp(t)| \leq C A^n (n \beta)^{n\beta}
\] 
for all integers $n\geq 0$ and $\hat \vp$ satisfies $L^\infty$-QSE of order $\beta$.
\item Suppose that $\hat\vp$ satisfies $L^1$-QSE of order $\beta$. Then there exist
$A,C>0$ so that
\[
|t^n \vp(t)| \leq C A^n (n\beta)^{n\beta}
\] 
for all integers $n\geq 0$, i.e.,
\[
| \vp(t)| \leq C e^{-a|t|^{1/\beta}}
\]
where $A = (\frac{\beta}{ae})^\beta$.
\end{enumerate}
\end{thm}

\begin{proof} Proof of (1). Recall that $\| \hat\vp^{(n)} \|_{L^\infty(\R)} \leq \| t^n \vp\|_{L^1(\R)}$. It follows that
\[
\| \hat\vp^{(n)} \|_{L^\infty(\R)}
\leq \int_{|t|\leq1} |t^n \vp(t)|\, dt + \int_{|t|>1} \frac 1{t^2} |t^{n+2} \vp(t)|\, dt
= 2C A^n n^{n \beta} (1 + A^2 n^{2\beta}).
\]
Next, if $A' = (1+\ep)A > A$, then for large enough $n$, $(A')^n \geq 1+ A^2n^{2\beta}$. Consequently, if 
$C' >C$ is sufficiently large,
\[
\| \hat \vp^{(n)} \|_{L^\infty(\R)}
\leq 2C A^n n^{n \beta} (1 + A^2 n^{2\beta}) 
\leq C' (A')^n n^{n\beta}.
\]

The proof of (2) is  immediate from the equality
$|t^n \vp(t)| = \big|\frac{1}{2\pi} \int_{\R} e^{it\tau} \hat\vp^{(n)}(\tau)\, d\tau \big|$.
\end{proof}

%
%
\section{Heat kernel decay estimates in terms of QSE}\label{sec:heat kernel in terms of QSE}
Fix $z,w\in \C$.  We are interested in the case for which
\[
d(z,w,t) \sim \mu\big(z,t+ T(w,z)\big).
\]
Since $\Box_b$ is translation invariant in $t$, if $\alpha=(z,t_1)$, $\beta = (w,t_2)$ and $t = t_1-t_2$, we can write 
$\H(s,\alpha,\beta) = \H(s,z,w,t)$.

We first prove the estimate (\ref{eqn:time derivative Boxb heat kernel decay}), 
with $J=J'=0$, for $\frac{\p\H}{\p s}(s,z,w,t)$ and then recover the
estimate for $\H(s,\alpha,\beta)$ from it.
We wish to find a sufficient condition so that
\begin{align}
\bigg| \frac{\p\H}{\p s} \big(s,z,w,t\big)\bigg | 
&\leq \frac{C}{s B_d(z,\sqrt s)} e^{-c \frac{\mu(z,t+T(w,z))^2}{s}} \nn\\
&= \sup_{j,k\geq 1} \frac{C}{B_d(z,\sqrt s)} \exp\Big(-\frac c{s|\ajkz|^{\frac 2{j+k}}} {|t+T(w,z)|^{\frac 2{j+k}}}\Big) \nn\\
&\sim \sum_{j,k\geq 1} \frac{C}{B_d(z,\sqrt s)} \exp\Big(-\frac c{s|\ajkz|^{\frac 2{j+k}}} {|t+T(w,z)|^{\frac 2{j+k}}}\Big). \label{eqn:decay we want for H in terms of z}
\end{align}
Since $\H(s,\alpha,\beta) = \overline{\H(s,\beta,\alpha)}$, we can interchange the roles of $z$ and $w$ in (\ref{eqn:decay we want for H in terms of z}) and 
we will find an estimate that implies \eqref{eqn:decay we want for H in terms of z}.
Let $\vp(t) = \frac{\p\H}{\p s}(s,z,w,t)$. 
By Corollary \ref{cor:exponential decay equivalence with estimates},
the exponential decay estimate (\ref{eqn:decay we want for H in terms of z}) 
is equivalent to the estimate
\[
|(t+ T(w,z))^n \vp(t)| \leq  \frac{C A^n}{sB_d(w,\sqrt s)} \sjk |\ajkw|^n s^{n\fjkt} n^{n\fjkt}
\]
for all $n\geq 0$.
We can incorporate the $sB_d(w,\sqrt s)$ into the sum by proving the following:
\begin{multline}
\label{eqn:equivalence}
\frac{1}{s^2} \sjk |\ajkw|^{n-1} s^{(n-1)\fjkt} n^{n\fjkt} \ges \frac{1}{sB_d(w,\sqrt s)} \sjk |\ajkw|^n s^{n\fjkt} n^{n\fjkt} \\
\ges \frac{1}{s^2 n^{\frac 12\deg p}} \sjk |\ajkw|^{n-1} s^{(n-1)\fjkt} n^{n\fjkt}
\end{multline}
where proportionality constants appearing in $\ges$ only depend on the 
the number of terms in the sum which is essentially the degree of the polynomial $p$. Also, since we are allowed geometric terms (i.e., $A^n$) and
$n^{\deg p}$ grows sub-geometrically, (\ref{eqn:equivalence}) allows us to
absorb $B_d(w,\sqrt s)$ into the sum.
To prove the inequalities, fix $s$ and observe that 
\begin{equation}
\label{eqn:denominator}
s B_d(w,\sqrt s) \sim s^2 \sjk |\ajkw| s^{\fjkt} \sim s^2 \max_{j,k\geq1} |\ajkw| s^{\fjkt}  =  s^2|\A{j_1k_1}w| s^{\frac{j_1+k_1}2}
\end{equation}
for some $j_1,k_1\geq 1$. Similarly, for each fixed $n$ (and $s$), 
\begin{equation}
\label{eqn:numerator}
\sjk |\ajkw|^n s^{n\fjkt} n^{n\fjkt} \sim \max |\ajkw|^n s^{n\fjkt} n^{n\fjkt} =  |\A{j_0k_0}w|^n s^{n\frac{j_0+k_0}2}
n^{n\frac{j_0+k_0}{2}}
\end{equation}
for some choice of index $j_0,k_0\geq 1$ (which depends on $n$ and $s$). 
From (\ref{eqn:denominator}), we have
$|\A{j_1k_1}w| s^{\frac{j_1+k_1}2} \geq |\A{j_0k_0}w| s^{\frac{j_0+k_0}2}$.
This inequality, together with (\ref{eqn:denominator}) and (\ref{eqn:numerator}) yield
\begin{eqnarray*}
\frac{1}{sB_d(w,\sqrt s)} \sjk |\ajkw|^n s^{n\fjkt} n^{n\fjkt}
&\sim& \frac{( |\A{j_0k_0}{w}| s^{\frac{j_0+k_0}2} )^n 
n^{n\frac{j_0+k_0}2}} {s^2|\A{j_1k_1}w| s^{\frac{j_1+k_1}2}} \\
&\leq& \frac{( |\A{j_0k_0}{w}| s^{\frac{j_0+k_0}2} )^{n-1} n^{n\frac{j_0+k_0}2}} {s^2}\\
&\leq& \frac{1}{s^2} \sum_{j,k \geq 1} |\ajkw|^{n-1} s^{(n-1)\frac{j+k}{2}} n^{n \frac{j+k}{2}}.
\end{eqnarray*}
This establishes that the first term (up a multiplicative constant) is larger than the second term in  
(\ref{eqn:equivalence}). 
To show that the second term is (up to a multiplicative
constant) larger than the third term in (\ref{eqn:equivalence}), we observe that
\begin{align*}
\label{eqn:reverse}
\frac{1}{sB_d(w,\sqrt s)} \sjk | \ajkw |^n s^{n\fjkt} n^{n\fjkt} 
&\ges \frac{ (\max \{| \ajkw | s^{\fjkt} n^{\fjkt}\})^n  } { s^2\max \{|\ajkw| s^{\fjkt} n^{\fjkt}\} } \\
& \sim \frac 1{s^2} \sjk| \ajkw |^{n-1} s^{(n-1)\fjkt} n^{(n-1)\fjkt} \\
&\geq \frac 1{s^2 n^{\frac 12 \deg p}} \sjk| \ajkw |^{n-1} s^{(n-1)\fjkt} n^{n\fjkt}.
\end{align*}
This establishes (\ref{eqn:equivalence}).

Thus, to show that $|\frac{ \p\H}{\p s}(s,z,w,t)|$ 
satisfies (\ref{eqn:time derivative Boxb heat kernel decay}), 
with $J=J'=0$,
we will show the equivalent condition that there exist constants $C,A>0$
so that 
\begin{equation}\label{eqn:estimate to show to prove theorem}
\|(t+ T(w,z))^n \vp\|_{L^\infty(\R)} \leq \|\Mtp^n \hat \vp\|_{L^1(\R)} \leq \frac{C A^n}{s^2} \sjk | \ajkw |^{n-1} s^{(n-1)\fjkt} n^{n\fjkt}.
\end{equation}

%
%
\section{Estimates for $\Mtp^n\Htp(s,z,w)$ and the proof of Theorem  \ref{thm:Boxb heat kernel decay} }
\label{sec:tau deriv estimates}
Since $\Boxtp$ is a self-adjoint operator in $L^2(\C)$, it follows
that $\Htp(s,z,w) = \overline{\Htp(s,w,z)}$ \cite{Rai06h}. 
Thus, the differential operators in $w$ are:
\begin{align*}
\Wbstpw &= \overline{(\Zstpw)}=\frac{\p}{\p \w} 
-  \tau \frac{\p p}{\p \w}= e^{\tau p}\frac{\p p}{\p \w}e^{-\tau p},
& \Wstpw &= \overline{(\Zbstpw)} = 
\frac{\p}{\p w} +  \tau \frac{\p p}{\p w}= e^{-\tau p}\frac{\p p}{\p w}e^{\tau p}.
\end{align*}

The goal of the remainder of the paper is to show the following theorem.
%
%
\begin{thm}\label{thm:Htp satisfies QSE}
Let $p:\C\to\R$ be a subharmonic, nonharmonic polynomial, $\tau>0$, and $n\geq 0$. Let $c_0$ be as in
(\ref{eqn:estimate for Htp}).
There exists constants $C>0$ so that
\begin{enumerate}\renewcommand{\labelenumi}{(\roman{enumi})}
\item 
\[
|(\Mtp^{z,w})^n\Htp(s,z,w)| \leq
\begin{cases} \displaystyle \frac{C^n}{s \tau^2} e^{-c_0\frac{|z-w|^2}{2s}}\sjk |\ajkz|^{n-2} s^{(n-2)\fjkt} n^{n\fjkt} \\
\displaystyle \frac{C^n}{s}e^{-c_0\frac{|z-w|^2}{2s}}\sjk |\ajkz|^n s^{n\fjkt} n^{n\fjkt}
\end{cases}
\]
\item If $X = \Zbstpz, \Zstpz, \Wbstpw$, or $\Wstpw$, then
\[
|X (\Mtp^{z,w})^n\Htp(s,z,w)| \leq
\begin{cases} \displaystyle \frac{C^n}{s^{3/2}\tau^2} e^{-c_0\frac{|z-w|^2}{2s}}\sjk |\ajkz|^{n-2} s^{(n-2)\fjkt} n^{n\fjkt} \\
\displaystyle \frac{C^n}{s^{3/2}}e^{-c_0\frac{|z-w|^2}{2s}}\sjk |\ajkz|^n s^{n\fjkt} n^{n\fjkt}
\end{cases}
\]
\item If
$X^2  = \Wstpw\Wbstpw$ or $X = \Zstpz\Wbstpw$, then
\[
|X^2 (\Mtp^{z,w})^n\Htp(s,z,w)| \leq
\begin{cases} \displaystyle \frac{C^n}{s^2\tau^2} e^{-c_0\frac{|z-w|^2}{4s}}\sjk |\ajkz|^{n-2} s^{(n-2)\fjkt} n^{n\fjkt} \\
\displaystyle \frac{C^n}{s^2}e^{-c_0\frac{|z-w|^2}{4s}}\sjk |\ajkz|^n s^{n\fjkt} n^{n\fjkt}
\end{cases}
\]
\end{enumerate}
\end{thm}

\begin{remark} The argument we give assumes $n\geq 3$. However, 
the $n\leq 2$ case follows from \cite{Rai10h}.
While the bounds in \cite{Rai10h} have better decay in $s$ and $|z-w|$ than in Theorem
\ref{thm:Htp satisfies QSE}, the constants depend on $n$ in an unknown way, hence we need the more careful argument presented here.

Also, observe that 
\[
n^{n\frac{j+k}2} 
\leq A^n (n-1)^{(n-1)\frac{j+k}2}, 
\]
for a suitable constant $A$.
This means that we have flexibility in the statement of Theorem \ref{thm:Htp satisfies QSE} in the sense that $(n-2)^{n-2}$ could be replaced by
$n^n$ (or $(n-1)^{n-1}$), etc. 
\end{remark}

\begin{rem} One trick that we use repeatedly  is the fact that for any $\ep>0$ and $n\geq 0$, 
there exists a constant
$C_{\ep,n}$ so that 
\begin{equation}
\label{eqn:al1}
e^{-c \frac ab} \leq C_{\ep,n} e^{-(1-\ep)c \frac ab} \frac{b^n}{a^n}.
\end{equation} 
We will use this inequality by either commenting we may need to decrease
$c$ for a subsequent inequality to hold true or we may simply and mysteriously halve the constant in the exponential.
\end{rem}

Theorem \ref{thm:Htp satisfies QSE} allows us to prove Theorem \ref{thm:Boxb heat kernel decay}. 
\begin{proof}[Proof of Theorem \ref{thm:Boxb heat kernel decay}.]
The reason that we estimate $|(\Mtpzw)^n\frac{\p \Htp}{\p s}(s,z,w)|$ first is that we can reduce the integral to the case when $\tau>0$. To see
how this works, we recall an observation from \cite{Rai06h}. Since $\Boxtp^k\Zbstp =  \Zbstp\Boxwtp^k$ for all $k\geq 0$, it follows that
\[
 e^{-s\Boxtp}\Zbstp =\Zbstp e^{-s\Boxwtp}.
\]
On the kernel side, if $dw$ is Lebesgue measure on $\R^2=\C$, then 
\[
e^{-s\Boxtp}\Zbstp \vp(z) = \int_\C \Htp(s,z,w)\Zbstpw\vp(w)\, dw = -\int_\C \Wbstpw\Htp(s,z,w)\vp(w)\, dw
\]
and
\[
\Zbstpz e^{-s\Boxwtp}\vp(z) = \int_\C \Zbstpz\Hwtp(s,z,w)\vp(w)\, dw.
\]
Thus,
\begin{equation}\label{eqn:relating H and H wiggle}
-\Wbstpw\Htp(s,z,w) = \Zbstpz\Hwtp(s,z,w).
\end{equation}
Since  $M_{-\tau p}^{z,w} = \overline{\Mtpzw}$,
by (\ref{eqn:HtwiddleH}) and (\ref{eqn:relating H and H wiggle}), 
we have (for $\tau>0$),
\begin{align}\label{eqn: M^n H tau<0 with M^n H tau>0}
\overline{(M_{-\tau p}^{z,w})^n \frac{\p H_{-\tau p}}{\p s}(s,z,w)}
=  (\Mtpzw)^n\frac{\p}{\p s} \Hwtp(s,z,w) 
&=   (\Mtpzw)^n \Zstpz\Zbstpz\Hwtp(s,z,w) \nn \\
&= - (\Mtpzw)^n \Zstpz\Wbstpw\Htp(s,z,w) 
\end{align}
As a consequence of (\ref{eqn: M^n H tau<0 with M^n H tau>0}), we have successfully reduced to the estimate on 
$|(\Mtpzw)^n\frac{\p \Htp}{\p s}(s,z,w)|$ for $\tau\in\R$ to an estimate on $|(\Mtpzw)^n\frac{\p \Htp}{\p s}(s,z,w)|$ for $\tau>0$.

With $X^2$ as in \emph{(iii)}, 
we need to  show that we can estimate of $(\Mtpzw)^nX^2\Htp(s,z,w)$ using Theorem \ref{thm:Htp satisfies QSE}. We handle one derivative
at a time. Assume that $X$ as in \emph{(ii)} of Theorem \ref{thm:Htp satisfies QSE}. 
Let $e(w,z) = \sum_{j\geq 1} \frac{1}{j!}\frac{\p^{j+1} p(z)}{\p z^j \p\z}(w-z)^j$. From Proposition 5.6
in \cite{Rai10h}, \[
e(w,z) = -\sum_{\atopp{j\geq 1}{k\geq 0}} \frac{1}{j!k!} \frac{\p^{j+k+1}p(w)}{\p w^j\p\w^{k+1}}(z-w)^j\overline{(z-w)}{}^k.
\]
and 
\begin{equation}
\label{eqn:al2}
e(w,z)= -[\Mtp,\Zbstp].
\end{equation}

We can write
\begin{align*}
|(\Mtpzw)^n X\Htp(s,z,w)| &\leq |X(\Mtpzw)^{n}\Htp(s,z,w)| + \sum_{j=0}^{n-1} |(\Mtpzw)^{n-1-j} [\Mtpzw,X] (\Mtpzw)^j\Htp(s,z,w)| \\
&= |X(\Mtpzw)^{n}\Htp(s,z,w)| + n|e(w,z)| |(\Mtpzw)^{n-1}\Htp(s,z,w)|
\end{align*}
Certainly, the only term to estimate is $|e(w,z)| |(\Mtpzw)^{n-1}\Htp(s,z,w)|$
Using \emph{(i)}, we have
\begin{align*}
|e(w,z)|& |(\Mtpzw)^{n-1}\Htp(s,z,w)| \\
&\leq \sab |\A{\alpha\beta}z| |w-z|^{\alpha+\beta-1} e^{-c_0 \frac{|z-w|^2}{2s}} \frac{C^{n-1}}{s} \sjk |\ajkz|^{n-1} s^{(n-1)\frac{j+k}2} (n-1)^{(n-1)\frac{j+k}2} ]\\
&\leq e^{-c_0 \frac{|z-w|^2}{4s}} \sab |\A{\alpha\beta}z| s^{\frac{\alpha+\beta-1}2}  \frac{C^{n-1}}{s} \sjk |\ajkz|^{n-1} s^{(n-1)\frac{j+k}2} (n-1)^{(n-1)\frac{j+k}2}\\
&\leq e^{-c_0 \frac{|z-w|^2}{4s}}   \frac{C^{n-1}}{s^{3/2}} \sjk |\ajkz|^{n} s^{n\frac{j+k}2} n^{n\frac{j+k}2}.
\end{align*}
Thus $(\Mtpzw)^n X\Htp(s,z,w)$ satisfies the estimate
\emph{(ii)} in Theorem \ref{thm:Htp satisfies QSE} for some uniform constant $c_0$.
By similar arguments, we can show that if $X^2$ is as in \emph{(iii)} of Theorem \ref{thm:Htp satisfies QSE}, then \break
$(\Mtpzw)^n X^2\Htp(s,z,w)$ satisfies the estimates given in \emph{(iii)} of Theorem \ref{thm:Htp satisfies QSE} when $\tau>0$ by cutting $c_0$ in half (again). 
Next, since $\frac{\p \Htp}{\p s}(s,z,w) = -\Wstpw\Wbstpw\Htp(s,z,w)$, it follows from the previous paragraph that
$(\Mtpzw)^n  \frac{\p \Htp}{\p s}(s,z,w)$ satisfies the estimates in \emph{(i)} of  Theorem \ref{thm:Htp satisfies QSE} for all $\tau$, 
both positive and negative
(up to a modification of $c_0$). 

Now we integrate this estimate in $\tau$. Observe that
\[
\int_0^\infty \min\{ |\ajkz|^2 s^{2\fjkt}, \tau^{-2}\}\, d\tau
= \int_0^{(|\ajkz|s^{\fjkt})^{-1}} |\ajkz|^2 s^{j+k}\, d\tau +
\int_{(|\ajkz|s^{\fjkt})^{-1}}^\infty \tau^{-2}\, d\tau = 2|\ajkz| s^{\fjkt}
\]
Using this together with the estimate for $(\Mtpzw)^n \Htp(s,z,w)$ in part \emph{(i)} of 
Theorem \ref{thm:Htp satisfies QSE} (for all $\tau$), we have
\[
\int_{-\infty}^\infty |(\Mtp^{z,w})^n\frac{\p\Htp}{\p s}(s,z,w)|\, d\tau
\leq \frac{C^n}{s^2}e^{-c_0 \frac{|z-w|^2}{4s}} \sjk |\ajkz|^{n-1} s^{(n-1)\fjkt} n^{n\fjkt}.
\]
By \eqref{eqn:estimate to show to prove theorem}, this proves the following estimate:
\begin{equation}
\label{eqn:al3}
\bigg|\frac{\p}{\p s}\H(s,\alpha,\beta)\bigg| 
\leq C \frac{e^{-c \frac{d(\alpha,\beta)^2}s}} {s |B_d(\alpha,\sqrt s)|}.
\end{equation}

To recover the estimate for $\H(s,\alpha,\beta)$, we use the Fundamental Theorem of Calculus
and the fact that $\H(0,\alpha,\beta)=0$ away from the diagonal. If we set $d = d(\alpha,\beta)$ and consider the case
$s \geq \frac12 c_0 d^2$, 
then we estimate (with replacing $c_0$ by a smaller
constant $c$ using (\ref{eqn:al1})),
\begin{align*}
|\H(s,\alpha,\beta)| &= \bigg| \int_0^s \frac{\p \H}{\p s}(r,\alpha,\beta)\, dr \bigg| \leq
\int_0^{d^2} \frac{C}{d^2(d^2\Lambda(z,d))} e^{-c \frac{d^2}s}\, dr + \int_{d^2}^s \frac{C}{r^2\Lambda(z,d)} e^{-c \frac{d^2}s}\, dr\\
&\leq \frac{C}{V(\alpha,\beta)} e^{-c \frac{d(\alpha,\beta)^2}s}.
\end{align*}
If $s \leq \frac 12 c_0 d^2$, then we estimate (with $c_0$ replaced by the smaller $c$  using (\ref{eqn:al1})) that
\[
|\H(s,\alpha,\beta)| = \bigg| \int_0^s \frac{\p \H}{\p s}(r,\alpha,\beta)\, dr \bigg| \leq
\int_0^{s} \frac{C}{r V(\alpha,\beta)} e^{-c \frac{d^2}{2r}}\, dr. 
\]
If we set $f(r) = \frac 1r e^{-c \frac{d^2}{2r}}$, then calculus shows that $f'(r)\geq 0$ when $r\leq \frac 12 c_0 d^2$. This means
\[
|\H(s,\alpha,\beta)| 
\leq
\frac{C}{V(\alpha,\beta)} \int_0^{s} f(r)\, dr\leq \frac{C}{V(\alpha,\beta)} s f(s) = \frac{C}{V(\alpha,\beta)} e^{-c_0 \frac{d^2}{2s}}.
\]

The passage from estimates on  $\H(s,\alpha,\beta)$ to estimates on $X^J_\alpha X^{J'}_\beta \H(s,\alpha,\beta)$ involves a short bootstrapping argument
and Theorem 3.4.2 from \cite{NaSt01h}, a Sobolev embedding theorem. Fix $s>0$ and $\beta\in\C\times\R$. We first bound derivatives
only in $\alpha$. From \cite{NaSt01f}, there exists
a bump function $\vp\in C^\infty_c(B_d(\alpha,\frac 12 d(\alpha,\beta)))$ so that $\vp(\gamma)=1$ on 
$B_d(\alpha, \frac 14 d(\alpha,\beta))$, $0\leq \vp\leq 1$ and for every multiindex $I$, $|X^I \vp| \leq \frac{c_{|I|}}{d(\alpha,\beta)^{|I|}}$ where $c_{|I|}$ is
independent of $\alpha$ and $d(\alpha,\beta)$. 
We now use Theorem 3.4.2 from \cite{NaSt01h} (and note that we may take $R_0=\infty$) and estimate that for some $C>0$ and $L\in\N$,
\begin{align}
|X^I_\alpha \H(s,\alpha,\beta)| &= |\vp(\alpha)X^I_\alpha \H(s,\alpha,\beta)| \nn \\
&\leq \frac{C}{V(\alpha,\beta)^{1/2}} \sum_{0 \leq|J|\leq L} d(\alpha,\beta)^{|J|} 
\big \|X^J_\alpha \big( \vp X^I_\alpha \H(s,\cdot ,\beta)\big) \big\|_{L^2(\C\times\R)}. \label{eqn:X^I H Sobolev}
\end{align}
The derivatives in this estimation are taken with respect to $\alpha$ and we will henceforth omit the subscript.
We integrate by parts using the fact that $(X^0)^* = -X^1$ (and $(X^1)^* = -X^0$) and obtain
\begin{align}
\big\| X^J \big( \vp X^I \H\big) \big\|_{L^2}^2
&= \la  X^J \big( \vp X^I \H\big),  X^J \big( \vp X^I \H\big) \ra
= \la   \H, (X^I)^*  \big( \vp  (X^J)^* X^J \big( \vp X^I \H\big)\big) \ra \nn \\
&\leq \| \H \|_{L^\infty(\supp\vp)} V(\alpha,\beta)^{1/2} \big\| (X^I)^* \big( \vp  (X^J)^* X^J \big( \vp X^I \H\big)\big) \big\|_{L^2}.
\label{eqn:X^J X^I H in L^2}
\end{align}
Since $d(\gamma,\beta) \geq \frac 12 d(\alpha,\beta)$ for $\gamma \in \supp\vp$,
\[
\big| (X^I)^*\big( \vp  (X^J)^* X^J \big( \vp X^I \H\big)\big) \big|
\leq C \sum_{|I_1|+|I_2|+|I_3|=2|I|+2|J|} |X^{I_1}\vp| |X^{I_2}\vp| |X^{I_3}_\gamma\H(s,\gamma,\beta)|,
\]
and from \cite{NaSt01h} or \cite{Rai10h}, 
$|X^{I_3}_\gamma\H(s,\gamma,\beta)| \leq C_{|I_3|} d(\gamma,\beta)^{-|I_3|} V(\gamma,\beta)^{-1}$, it follows that
\begin{align*}
\big\| (X^I)^*  \big( \vp  (X^J)^* X^J \big( \vp X^I \H\big)\big) \big\|_{L^2}
&\leq \big\| (X^I)^* \big( \vp  (X^J)^* X^J \big( \vp X^I \H\big)\big) \big\|_{L^\infty}V(\alpha,\beta)^{1/2}\\
&\leq \frac{C}{d(\alpha,\beta)^{2|I|+2|J|} V(\alpha,\beta)^{1/2}}.
\end{align*}
Using the estimate on $\H(s,\gamma,\beta)$ proven above, we have that on $\supp\vp$, 
$|\H(s,\gamma,\beta)| \leq C \frac{e^{-c \frac{d(\alpha,\beta)^2}{2s}}}{V(\alpha,\beta)}$, so plugging our estimate on 
$\big\| X^I  \big( \vp  X^J X^J \big( \vp X^I \H\big)\big) \big\|_{L^2}$ into \eqref{eqn:X^J X^I H in L^2}
and that into \eqref{eqn:X^I H Sobolev}, we get (with a further decrease in $c$) that
\begin{align*}
|X^I_\alpha \H(s,\alpha,\beta)|
&\leq \frac{C}{V(\alpha,\beta)^{1/2}} \sum_{0 \leq|J|\leq L} d(\alpha,\beta)^{|J|} \frac{e^{-c \frac{d(\alpha,\beta)^2}s}}{V(\alpha,\beta)^{1/2}V(\alpha,\beta)^{1/4}}
\frac{1}{d(\alpha,\beta)^{|I|+|J|}} V(\alpha,\beta)^{1/4} \\
&\leq \frac{C}{d(\alpha,\beta)^{|I|}} \frac{e^{-c \frac{d(\alpha,\beta)^2}s}}{V(\alpha,\beta)},
\end{align*}
the desired estimate. To pass from estimates on $X^J_\alpha\H(s,\alpha,\beta)$ to estimates on $X^J_\alpha X^{J'}_\beta\H(s,\alpha,\beta)$, we simply
repeat the argument in $\beta$ with  $X^J_\alpha \H(s,\alpha,\beta)$ playing the role of $\H(s,\alpha,\beta)$. Finally, since
$\frac{\p^j}{\p s^j}\H(s,\alpha,\beta) = (-1)^j\Box_b^j \H(s,\alpha,\beta)$, proving the estimates for $X^J_\alpha X^{J'}_\beta\H(s,\alpha,\beta)$ is 
sufficient to prove the theorem.
\end{proof}

\begin{rem} The estimates  in Theorem \ref{thm:Htp satisfies QSE} allow us to prove that
$e^{-i\tau T(w,z)}\Htp(s,z,w)$ satisfies $L^1$-QSE for every $\beta = j+k$ where $j\geq 1$, $k\geq 1$ (of course, for 
$j+k\geq \deg p$, the condition is vacuous). The exponential decay for $H(s,\alpha,\beta)$ follows by proving the $L^1$-QSE and
keeping careful track of the powers of $s$ and $|\ajko|$. 
\end{rem}

%
%
\section{A good formula for $\Mtpzw\Htp(s,z,w)$}
\label{sec:a good formula for MH}
The goal of this section is to prove a tractable formula for $\Mtpzw\Htp(s,z,w)$.
The launching point is the solution to the nonhomogeneous heat equation in \cite{Rai10h} given by a 
Duhamel's formula.

Proposition 5.1 in \cite{Rai10h} yields
\begin{prop}\label{prop:IVP} 
Let $g:(0,\infty)\times\C\to\C$ and $f:\C\to\C$ be $L^2(\C)$ 
for each $s$ and vanish as
$|z|\to\infty$.
The solution to the nonhomogeneous heat equation
\begin{equation}\label{eqn:nonhomog IVP}
\begin{cases} \displaystyle \frac{\p u}{\p s} + \Boxtp u = g \text{  in } (0,\infty)\times\C
\vspace{.075in}\\
\displaystyle \lim_{s\to 0} u(s,z) = f(z) \end{cases}
\end{equation}
is given by
\[
u(s,z) = \int_\C \Htp(s,z,\xi) f(\xi)\, d\xi + \int_0^s \int_\C \Htp(s-r,z,\xi) g(r,\xi)\, d\xi dr.
\]
\end{prop}

Observe that $(\Mtp^{z,w})^n\Htp(s,z,w)$ behaves as follows.
\begin{align*}
g&:= \Big(\frac{\p}{\p s} + \Boxtpz \Big)(\Mtp^{z,w})^n\Htp(s,z,w)
= \Mtp^n \frac{\p \Htp}{\p s} + \Mtp\Boxtp\Mtp^{n-1}\Htp + [\Boxtp,\Mtp]\Mtp^{n-1}\Htp \\
&= \Mtp^n \frac{\p \Htp}{\p s} + \Mtp^2\Boxtp\Mtp^{n-2}\Htp + \Mtp[\Boxtp,\Mtp]\Mtp^{n-2}\Htp + [\Boxtp,\Mtp]\Mtp^{n-1}\Htp \\
&= \cdots = \underbrace{ \Mtp^n \frac{\p \Htp}{\p s} + \Mtp^n \Boxtp\Htp}_{=0}
+ \Mtp^{n-1}[\Boxtp,\Mtp]\Htp + \Mtp^{n-2}[\Boxtp,\Mtp]\Mtp\Htp \\
&\hspace{2.5in} + \cdots \Mtp[\Boxtp,\Mtp]\Mtp^{n-2} + [\Boxtp,\Mtp]\Mtp^{n-1}\Htp.
\end{align*}
From \cite{Rai10h}, Proposition 5.4 we have
\begin{multline}\label{eqn:Box M commutator}
[\Boxtp,\Mtpzw] = \Zbstpz \overline{e(w,z)} - e(w,z)\Zstpz = \Zbstpz \overline{e(w,z)} - \Zstpz e(w,z) + \frac{\p^2 p(z)}{\p z \p\z} \\
= -\frac{\p^2 p}{\p z\p\z} - e(w,z)\Zstpz + \overline{e(w,z)}\Zbstpz.
\end{multline}
To simplify the calculation further,  observe
\begin{align*}
\Mtp[\Boxtp,\Mtp] &= [\Boxtp,\Mtp]\Mtp + \big[\Mtp,[\Boxtp,\Mtp]\big] \\
&= [\Boxtp,\Mtp]\Mtp + \big[\Mtp,-\frac{\p^2 p}{\p z\p\z} - e(w,z)\Zstpz + \overline{e(w,z)}\Zbstpz] \\
&= [\Boxtp,\Mtp]\Mtp - e(w,z)[\Mtp,\Zstp] + \overline{e(w,z)}[\Mtp,\Zbstp] \\
&= [\Boxtp,\Mtp]\Mtp - 2|e(w,z)|^2.
\end{align*}
where the next to last equality uses (\ref{eqn:al2}).
Consequently,
\begin{align*}
\Mtp^j[\Boxtp,\Mtp] &= \Mtp^{j-1}(\Mtp[\Boxtp,\Mtp])= \Mtp^{j-1}\big( [\Boxtp,\Mtp]\Mtp - 2|e(w,z)|^2\big) \\
&= \Mtp^{j-2} [\Boxtp,\Mtp]\Mtp^2 - 2\Mtp^{j-2}|e(w,z)|^2 \Mtp - 2 \Mtp^{j-1}|e(w,z)|^2 \\
&= \Mtp^{j-2} [\Boxtp,\Mtp]\Mtp^2 - 4|e(w,z)|^2\Mtp^{j-1} \\
&= \cdots = [\Boxtp,\Mtp]\Mtp^j - 2j|e(w,z)|^2 \Mtp^{j-1}.
\end{align*}
Therefore,
\begin{align*}
g=\sum_{j=0}^{n-1} \Mtp^j [\Boxtp,\Mtp]\Mtp^{n-1-j}\Htp
&= \sum_{j=0}^{n-1}\Big( [\Boxtp,\Mtp]\Mtp^j - 2j|e(w,z)|^2\Mtp^{j-1}\Big) \Mtp^{n-1-j}\Htp \\
&= \sum_{j=0}^{n-1}[\Boxtp,\Mtp]\Mtp^{n-1}\Htp - |e(w,z)|^2\Mtp^{n-2}\Htp\sum_{j=1}^{n-1}2j \\
&= n[\Boxtp,\Mtp]\Mtp^{n-1}\Htp - n(n-1)|e(w,z)|^2\Mtp^{n-2}\Htp.
\end{align*}
 
From Theorem 6.3 in \cite{Rai10h}, it follows that the single integral term in Proposition \ref{prop:IVP} is 0, so we have: 
\begin{prop}\label{prop:n derivs, no recursion}
\begin{multline}\label{eqn:n derivs, no recursion}
(\Mtp^{z,w})^n\Htp(s,z,w) = n\int_0^s\int_\C \Htp(s-r,z,\xi) [\Boxtpxi,\Mtp^{\xi,w}] (\Mtp^{\xi,w})^{n-1}\Htp(r,\xi,w)\,
d\xi\, dr \\
-n(n-1) \int_0^s\int_\C \Htp(s-r,z,\xi) |e(w,\xi)|^2 (\Mtp^{\xi,w})^{n-2}\Htp(r,\xi,w)\,
d\xi\, dr.
\end{multline}
\end{prop}

We use Proposition \ref{prop:n derivs, no recursion} as a starting point for a recursion to generate a formula for
$(\Mtp^{z,w})^n\Htp(s,z,w)$ that involves no $\tau$-derivatives of $\Htp(s,z,w)$. Plugging
the integral for $(\Mtp^{z,w})^{n-1}\Htp(s,z,w)$ and $(\Mtp^{z,w})^{n-2}\Htp(s,z,w)$ into the RHS of \eqref{eqn:n derivs, no recursion},
we have
\begin{align*}
&(\Mtp^{z,w})^n\Htp(s,z,w) \\
&= n(n-1)\int_0^s \int_0^{r_1} \int_{\C^2} \Htp(s-r_1,z,\xi_1) 
[\Box_{\tau p,\xi_1},\Mtp^{\xi_1,w}]\Htp(r_1-r_2,\xi_1,\xi_2) \\
&\hspace{2in}\times [\Box_{\tau p,\xi_2},\Mtp^{\xi_2,w}] 
(\Mtp^{\xi_2,w})^{n-2} \Htp(r_2,\xi_2,w)\, d\xi_2\,d\xi_1\, dr_2 \, dr_1 \\
&- n(n-1)(n-2) \int_0^s \int_0^{r_1} \int_{\C^2} \Htp(s-r_1,z,\xi_1) 
[\Box_{\tau p,\xi_1},\Mtp^{\xi_1,w}]\Htp(r_1-r_2,\xi_1,\xi_2)\\
&\hspace{2in}\times |e(w,\xi_2)|^2 (\Mtp^{\xi_2,w})^{n-3} \Htp(r_2,\xi_2,w)\,
d\xi_2\,d\xi_1\, dr_2 \, dr_1 \\
&- n(n-1)(n-2) \int_0^s \int_0^{r_1} \int_{\C^2} \Htp(s-r_1,z,\xi_1) 
|e(w,\xi_1)|^2 \Htp(r_1-r_2,\xi_1,\xi_2)\\
&\hspace{2in}\times [\Box_{\tau p,\xi_1},\Mtp^{\xi_2,w}] (\Mtp^{\xi_2,w})^{n-3} \Htp(r_2,\xi_2,w)\,
d\xi_2\,d\xi_1\, dr_2 \, dr_1 \\
&+ n(n-1)(n-2)(n-3) \int_0^s \int_0^{r_1} \int_{\C^2} \Htp(s-r_1,z,\xi_1) 
|e(w,\xi_1)|^2 \Htp(r_1-r_2,\xi_1,\xi_2)\\
&\hspace{2in}\times |e(w,\xi_2)|^2 (\Mtp^{\xi_2,w})^{n-4} \Htp(r_2,\xi_2,w)\,
d\xi_2\,d\xi_1\, dr_2 \, dr_1
\end{align*}
The procedure is repeated while there are still $\Mtp\Htp$ terms left in the integrals. To calculate the resulting
integral, a number of observations are needed. First, since the integral for $n$-derivatives decomposes to a sum involving
$(n-1)$-derivatives and $(n-2)$-derivatives, if $f_n$ is the number of integrals that $n$-derivatives decomposes into, then
we have the relation
\[
f_n = f_{n-1}+f_{n-2}.
\]
Also, we know that $f_1=1$ and by Proposition \ref{prop:n derivs, no recursion}, $f_2 =2$. 
Thus, $f_n$ is the $n$th Fibonacci number and
\[
f_n = \frac{1}{\sqrt 5}\bigg( \Big(\frac{1+\sqrt 5}2\Big)^{n+1} - \Big(\frac{1-\sqrt 5}2\Big)^{n+1}\bigg). 
\]
The important feature of $f_n$ is that it grows geometrically with $n$ (and not faster!). It is easiest to
describe the derivation for the formula for $(\Mtp^{z,w})^n\Htp(s,z,w)$ in the language of trees. 
The descendants of $(\Mtp^{z,w})^n\Htp(s,z,w)$ are an integral that involves $(\Mtp^{z,w})^{n-1}\Htp(s,z,w)$
and an integral that involves $(\Mtp^{z,w})^{n-2}\Htp(s,z,w)$. The child that inherits the term with
$(\Mtp^{z,w})^{n-1}\Htp(s,z,w)$  comes with a factor $n$ and the commutator $[\Boxtp,\Mtp]$. The child
with $(\Mtp^{z,w})^{n-2}\Htp(s,z,w)$ inherits a factor of $-n(n-1)$ and an $|e(w,\xi)|^2$-term. We know that there
are $f_n$ paths down the tree. Let the left child denote the term where $\Mtp$ drops by one degree
and the right child denote the term where $\Mtp$ drops by two degrees.
Let $\I_n$ denote the set of paths down tree for $(\Mtp^{z,w})^{n}\Htp(s,z,w)$. 
A path $J\in \I_n$ is a sequence $\{a_j\}$ with $a_j=1$ indicating a ``left" child and $a_j=2$ indicating a ``right" child.
The path length is $|J|$. It follows that $n/2 \leq |J| \leq n$. Let $J_1 = \#\{j\in J : a_j=1\}$ and
$J_2 = \#\{j\in J: a_j=2\}$. 
Let
\[
N(a_j,\xi_j) = \begin{cases} [\Box_{\tau p,\xi_j},\Mtp^{\xi_j,w}] & a_j=1 \\ |e(w,\xi_j)|^2 & a_j=2\end{cases}.
\]
The operator $N(a_j,\xi_j)$ records the information discussed above. It follows that 
\begin{prop}\label{prop:n derivs, recursion}
\begin{align*}
&(\Mtp^{z,w})^n\Htp(s,z,w)\\
&= n! \sum_{J\in\I_n} (-1)^{J_2} \int_0^s \int_0^{r_1} \cdots \int_0^{r_{|J|-1}} \int_{\C^{|J|}}
\Htp(s-r_1,z,\xi_1) \Big( \prod_{j=1}^{|J|-1} N(a_j,\xi_j) \Htp(r_j-r_{j+1},\xi_j,\xi_{j+1})\Big) \\
&\times N(a_{|J|},\xi_{|J|}) \Htp(r_{|J|},\xi_{|J|},w)\, d\xi_{|J|}\cdots d\xi_1\, dr_{|J|}\cdots dr_1.
\end{align*}
\end{prop}

%
%
\section{Proof of Theorem \ref{thm:Htp satisfies QSE}}
\label{sec: proof of QSE thm}
Understanding how to manipulate the formula in Proposition \ref{prop:n derivs, recursion} is the crux of the proof. The three parts of
Theorem \ref{thm:Htp satisfies QSE} are proven similarly, though not identically. We will start with \emph{(i)} and prove it in detail. We will discuss
the modifications necessary for \emph{(ii)} and \emph{(iii)}. The workhorse estimates for proving Theorem \ref{thm:Htp satisfies QSE} are the following
estimates from \cite{Rai06h}.
When $\tau>0$, 
\begin{equation}\label{eqn:estimate for Htp}
|\Htp(s,z,w)| \leq \frac{C}s e^{-c_0 \frac{|z-w|^2}s} e^{-c_0 \frac{s}{\mu(z,1/\tau)^2}}e^{-c_0 \frac{s}{\mu(w,1/\tau)^2}} .
\end{equation}
and
\begin{equation}
|\Zbstpz\Htp(s,z,w)| + |\Zstp\Htp(s,z,w)|
\leq \frac{C}{s^{3/2}} e^{-c_0 \frac{|z-w|^2}s} e^{-c_0 \frac{s}{\mu(z,1/\tau)^2}}e^{-c_0 \frac{s}{\mu(w,1/\tau)^2}}.
\label{eqn:Zbstp Htp estimates}
\end{equation}
\begin{rem}
When $\tau <0$, 
$\Htp(s,z,w)$ satisfies a weaker estimate (proven in \cite{Rai07}).
Fortunately, we avoid this difficulty here by
we exploiting the equality $\Box_{-\tau p} = \overline{\Boxwtp}$ and the fact that we can write
certain derivatives of $\Hwtp(s,z,w)$ in terms of $\Htp(s,z,w)$ as done in (\ref{eqn: M^n H tau<0 with M^n H tau>0}).
\end{rem}

Since there are only $f_n$-terms in the calculation and $f_n$ grows geometrically with $n$, we can treat each
integral from Proposition \ref{prop:n derivs, recursion} separately. The integrals can all be handled analogously, and we choose
to show a specific one for expositional clarity. We will show the case when $a_j=1$ for all $j$. Even more specifically,
$[\Boxtp,\Mtp]$, as given in the second line 
of (\ref{eqn:Box M commutator}), contains three terms.  We concentrate on the term that always has $e(w,\xi)\Zstpxi$. 
Without loss of generality, we can take $w=0$ since the argument is the same regardless of the $w$ we choose. 
The integral we estimate is
\begin{multline*}
I := \Bigg| \int_0^s \int_0^{r_1} \cdots \int_0^{r_{n-1}} \int_{\C^{n}} \Htp(s-r_1,z,\xi_1)
\Big( \prod_{j=1}^{n-1} e(0,\xi_j) Z_{\tau p,\xi_j} \Htp(r_j-r_{j+1},\xi_j,\xi_{j+1})\Big) \\
e(0,\xi_n) Z_{\tau p,\xi_n} \Htp(r_n,\xi_n,0)\, d\xi_n\cdots d\xi_1\, dr_n\cdots  dr_1\Bigg|.
\end{multline*}

The following inequality follows from the concavity of the logarithm and the convexity of $x^k$.
\begin{lem}\label{lem:sum to a power}
Let $ k$ be a positive integer and $a_1,\dots,a_{k}>0$. Then
\[
(a_1 \cdots a_k) \leq \frac{1}{k}(a_1^k + \cdots + a_k^k).
\] 
\end{lem}
The inequality is seen to be sharp by considering $a_1= \cdots = a_k =a$.
The other extremely useful fact is that
\begin{multline}\label{eqn:gauss convolution}
\exp\Big(-c_0 \frac{|\xi_{k-1}-\xi_k|^2}{r_{k-1}-r_k}\Big)\exp\Big(-c_0 \frac{|\xi_{k}|^2}{r_{k}}\Big) \\
= \exp\Big(-c_0 \frac{r_{k-1}}{(r_{k-1}-r_k)r_{k}}
 \Big| \xi_k - \frac{r_{k}}  {r_{k-1}} \xi_{k-1}\Big|^2\Big) \exp\Big(-c_0 \frac{|\xi_{k-1}|^2}{r_{k-1}}\Big). 
\end{multline}

We now start the proof of the estimates of Theorem 
\ref{thm:Htp satisfies QSE}. First, we will handle the estimates without
the term $1/\tau^2$ on the right; these will be referred to as the estimates without $\tau$-decay.
Then the argument will then be modified to establish the estimates with  $1/\tau^2$ on the right,
and these will be referred to as the estimates with $\tau$-decay. 

\subsection{Estimate \emph{(i)} of $I$ without $\tau$-decay.}
By Lemma \ref{lem:sum to a power}, we have
\[
|e(0,\xi_1)\cdots e(0,\xi_n)| \leq \frac 1n\big( |e(0,\xi_1)|^n + \cdots + |e(0,\xi_n)|^n\big).
\]
We let $C = C(p)$ (or $A$) be a constant that may vary from line to line and may depend on
$\deg(\triangle p)+2$ but NOT on $n$, the coefficients of $p$, or $s$. By (\ref{eqn:estimate for Htp})
and (\ref{eqn:Zbstp Htp estimates}), we have
\begin{multline*}
I \leq \frac{A^n}n \int_0^s\cdots\int_0^{r_{n-1}}\int_{\C^n}
\frac{e^{-c_0 \frac{|z-\xi_1|^2}{s-r_1}}}{s-r_1}
\Big(\frac{\prod_{m=1}^{n-1} e^{-c_0 \frac{|\xi_m-\xi_{m+1}|^2}{r_m-r_{m+1}}} }
{(r_m-r_{m+1})^{3/2}}\Big) 
\frac{e^{-c_0}\frac{|\xi_n|^2}{r_n}}{r_n^{3/2}}\\
\big( |e(0,\xi_1)|^n + \cdots + |e(0,\xi_n)|^n\big)\, d\xi_n\cdots d\xi_1\, dr_n\cdots dr_1. 
\end{multline*}
Note that we have ignored the terms on the right in (\ref{eqn:estimate for Htp})
and (\ref{eqn:Zbstp Htp estimates}) that involve decay in $\tau$ for this part of the argument.
Choosing an arbitrary $e(0,\xi_\ell)$ term, we estimate  the space integral first. Also, set $r_0=s$.
\begin{align*}
&\int_{\C^n}e^{-c_0 \frac{|z-\xi_1|^2}{s-r_1}}  
\Big[\prod_{m=1}^{n-1} e^{-c_0 \frac{|\xi_m-\xi_{m+1}|^2}{r_m-r_{m+1}}}\Big]e^{-c_0\frac{|\xi_n|^2}{r_n}}
|e(0,\xi_\ell)|^n \, d\xi_n\cdots d\xi_1 \\
&\leq C \int_{\C^n}e^{-c_0 \frac{|z-\xi_1|^2}{s-r_1}}  
\Big[\prod_{m=1}^{n-1} e^{-c_0 \frac{|\xi_m-\xi_{m+1}|^2}{r_m-r_{m+1}}}\Big]e^{-c_0\frac{|\xi_n|^2}{r_n}}
\big(\sumjko |\ajkoo| |\xi_\ell |^{j+k} \big)^n \, d\xi_n\cdots d\xi_1 \\
&\leq A^n\sumjko |\ajkoo|^n \int_{\C^n}e^{-c_0 \frac{|z-\xi_1|^2}{s-r_1}}  
\Big[\prod_{m=1}^{n-1} e^{-c_0 \frac{|\xi_m-\xi_{m+1}|^2}{r_m-r_{m+1}}}\Big]e^{-c_0\frac{|\xi_n|^2}{r_n}}
|\xi_\ell |^{n(j+k)}  \, d\xi_n\cdots d\xi_1 \\
&\leq A^n e^{-\frac{c_0}2 \frac{|z|^2}{s}} \sumjko |\ajkoo|^n \int_{\C^n}  
\Big[\prod_{m=1}^{\ell} e^{-\frac{c_0}2 \frac{r_{m-1}}{(r_{m-1}-r_m)r_m} 
|\xi_m- \frac {r_m}{r_{m-1}} \xi_{m-1}|^2} \Big]  \\
&\hspace{.5in}
\Big(\prod_{\alpha=\ell+1}^n e^{-c_0 \frac{r_{\alpha-1}}{(r_{\alpha-1}-r_\alpha)r_\alpha}
|\xi_\alpha- \frac {r_\alpha}{r_{\alpha-1}} \xi_{\alpha-1}|^2} \Big)
e^{-\frac{c_0}2\frac{|\xi_\ell|^2}{r_\ell}} |\xi_\ell |^{n(j+k)} \, d\xi_n\cdots d\xi_1,
\end{align*} 
where the last inequality uses (\ref{eqn:gauss convolution}) repeatedly. By \eqref{eqn:polynomial times exponential decay}, 
\[
e^{-\frac{c_0}2\frac{|\xi_\ell|^2}{r_\ell}} |\xi_\ell |^{n(j+k)}
\leq \Big( \frac{(j+k) r_\ell}{e c_0}\Big)^{n\fjkt} n^{n\fjkt}
\leq A^n s^{n\fjkt} n^{n\fjkt}
\]
for all $1 \leq j+k \leq \textrm{deg} (p)$.
Consequently,
\begin{align*}
&\int_{\C^n}e^{-c_0 \frac{|z-\xi_1|^2}{s-r_1}}  
\Big[\prod_{m=1}^{n-1} e^{-c_0 \frac{|\xi_m-\xi_{m+1}|^2}{r_m-r_{m+1}}}\Big]e^{-c_0\frac{|\xi_n|^2}{r_n}}
|e(0,\xi_\ell)|^n \, d\xi_n\cdots d\xi_1 \\
&\leq A^n e^{-\frac{c_0}2 \frac{|z|^2}{s}} \sumjko\big( |\ajkoo|^n s^{n\fjkt} n^{n\fjkt} \big)
\int_{\C^n}  
\Big[\prod_{m=1}^{\ell} e^{-\frac{c_0}2 \frac{r_{m-1}}{(r_{m-1}-r_m)r_m} 
|\xi_m- \frac {r_m}{r_{m-1}} \xi_{m-1}|^2} \Big]  \\
&\hspace{.5in}
\Big(\prod_{\alpha=\ell+1}^n e^{-c_0 \frac{r_{\alpha-1}}{(r_{\alpha-1}-r_\alpha)r_\alpha}
|\xi_\alpha- \frac {r_\alpha}{r_{\alpha-1}} \xi_{\alpha-1}|^2} \Big) 
\, d\xi_n\cdots d\xi_1\\
&= A^n e^{-\frac{c_0}2 \frac{|z|^2}{s}} \sumjko\big( |\ajkoo|^n s^{n\fjkt} n^{n\fjkt} \big)
\frac{(s-r_1)r_1}s \Big(\prod_{m=1}^{n-1} \frac{ (r_m-r_{m+1})r_{m+1}}{r_m}\Big).
\end{align*} 
Plugging this space integral estimate into the estimate for $I$, we have
\begin{multline*}
I \leq \frac{A^n}{sn} e^{-\frac{c_0}2 \frac{|z|^2}{s}} \sumjko\big( |\ajkoo|^n s^{n\fjkt} n^{n\fjkt} \big) \\
\times 
\int_0^{s}\int_{0}^{r_1}\cdots\int_0^{r_{n-1}}
(r_1-r_2)^{-1/2}\cdots (r_{n-1}-r_n)^{-1/2} r_n^{-1/2}\, dr_n\cdots dr_1.
\end{multline*}
To estimate the convolutions in the time (i.e., $r$-integrals), we use the $\beta$-function result
\begin{equation}\label{eqn:convolution of s and (r-s)}
\int_0^r \frac{s^{m/2-1}}{(r-s)^{1/2}}\, ds
= r^{\frac{m+1}2 -1} \int_0^1 s^{\frac{m}2 -1} (1-s)^{\frac12-1} \, ds
=  r^{\frac{m+1}2 -1} \sqrt\pi \frac{\Gamma(\frac m2)} {\Gamma(\frac{m+1}2)}.
\end{equation}
Thus,
\begin{align*}
\int_0^s \int_{0}^{r_1}\cdots&\int_0^{r_{n-1}}
(r_1-r_2)^{1/2-1}\cdots (r_{n-1}-r_n)^{1/2-1} r_n^{1/2-1}\, dr_n\cdots dr_1\\
&= \sqrt\pi \frac{\Gamma(\frac 12)}{\Gamma(\frac 22)}
\int_0^s \int_{0}^{r_1}\cdots\int_0^{r_{n-2}}
(r_1-r_2)^{1/2-1}\cdots (r_{n-2}-r_{n-1})^{1/2-1}r_{n-1}^{1-1} \, dr_{n-1}\cdots dr_1 \\
&= \pi^{2/2}  \frac{\Gamma(\frac 12)}{\Gamma(\frac 22)} \frac{\Gamma(\frac 22)}{\Gamma(\frac 32)}
\int_0^s \int_{0}^{r_1}\cdots\int_0^{r_{n-3}}
(r_1-r_2)^{1/2-1}\cdots (r_{n-3}-r_{n-2})^{1/2-1}r_{n-2}^{\frac 32-1} \, dr_{n-2}\cdots dr_1 \\
&= \cdots = \pi^{\frac{n-1}2} \frac{\Gamma(1/2)}{\Gamma(n/2)}
\int_0^s r_1^{\frac n2 -1}\, dr_1 
= \frac{ \pi^{n/2}}{\frac n2 \Gamma(\frac n2)} s^{n/2}.
\end{align*}
Combining our estimates together, we have
\[
I \leq \frac{A^n}{s n^2 \Gamma(n/2)}
e^{-\frac{c_0}2 \frac{|z|^2}{s}} \sumjko\big( |\ajkoo|^n s^{n\frac{j+(k+1)}2} n^{n\fjkt} \big).
\]
It turns out that the $n\Gamma(n/2)$ term is exactly what we need to attain
(\ref{eqn:estimate to show to prove theorem}). In the statement of Proposition \ref{prop:n derivs, recursion}, there
is an $n!$ multiplying the integral. By Stirling's formula, we can bound
\[
\frac{n!}{n^2 \Gamma(n/2)} \leq A^n \frac{n^n}{n^{n/2}} = A^n n^{n/2}.
\]
Therefore,
\[
n^{n\fjkt} \frac{n!}{n^2 \Gamma(n/2)} \leq A^n n^{n\frac{j+(k+1)}2}
\]
Reindexing our sum and interchanging $z$ and $w$, we have shown that for $\tau>0$,
\[
|(\Mtp^{z,w})^n\Htp(s,z,w)|
\leq \frac{A^n}s e^{-c_0\frac{|z|^2}{2s}} \sjk |\ajkz|^n s^{n\fjkt} n^{n\fjkt}
\]
which is the desired estimate in (\emph{i}) without decay.

\subsection{Second estimation of $I$ with decay.}
\label{subsec:al4}
This time we will exploit the $\tau$ decay terms 
in (\ref{eqn:estimate for Htp}) and (\ref{eqn:Zbstp Htp estimates}),
(including those depending on $\mu (\xi, 1/\tau)$). 
We also apply Lemma \ref{lem:sum to a power} to $\prod_{j=1}^{n-2} |e(0, \xi_j)|$
leaving $|e(0,\xi_{n-1})||e(0, \xi_n)|$ alone. We obtain
\begin{multline*}
I \leq \frac{A^{n-2}}{(n-2)} \int_0^s\cdots\int_0^{r_{n-1}}\int_{\C^n}
\frac{e^{-c_0 \frac{|z-\xi_1|^2}{s-r_1}}}{s-r_1}
\Big( \prod_{j=1}^{n-1} \frac{e^{-c_0 \frac{|\xi_j-\xi_{j+1}|^2} {r_j-r_{j+1}}}}
{(r_j-r_{j+1})^{3/2}}\Big) 
\frac{e^{-c_0}\frac{|\xi_n|^2}{r_n}}{r_n^{3/2}}\\
\times \big( |e(0,\xi_1)|^{n-2} + \cdots + |e(0,\xi_{n-2})|^{n-2}\big)|e(0,\xi_{n-1})||e(0,\xi_n)| \\
\times e^{-c_0 \frac{r_{n-2}-r_n}{\mu(\xi_{n-1},1/\tau)^2}}e^{-c_0 \frac{r_{n-1}}{\mu(\xi_{n},1/\tau)^2}}
e^{-c_0 \frac{r_n}{\mu(0,1/\tau)^2}}
\, d\xi_n\cdots d\xi_1\, dr_n\cdots dr_1. 
\end{multline*}
In the above calculation, we used 
$\exp(-c_0 \frac{r_{n-2}-r_{n-1}}{\mu(\xi_{n-1},1/\tau)})\exp(-c_0 \frac{r_{n-1}-r_n}{\mu(\xi_{n-1},1/\tau)})
= \exp(-c_0 \frac{r_{n-2}-r_n}{\mu(\xi_{n-1},1/\tau)})$,
which explains the appearance of this term in the above integrand.

We pick just one $|e(0,\xi_\ell)|^{n-2}$ term and concentrate on the space integral. 
Using (\ref{eqn:polynomial times exponential decay}), we have
\[
e^{-c_0 \frac{|\xi_\ell|^2}{r_\ell}} |e(0,\xi_\ell)|^{n-2}
\leq \sumjko |\ajkoo|^{n-2} s^{(n-2)\fjkt}(n-2)^{(n-2)\fjkt}
\]
(since $r_\ell \leq s$).
By a repeated use of (\ref{eqn:gauss convolution}) as we did in 
our first estimate of $I$, we have
\begin{align*}
&II := \int_{\C^n}
e^{-c_0 \frac{|z-\xi_1|^2}{s-r_1}}
\Big( \prod_{j=1}^{n-1} e^{-c_0 \frac{|\xi_j-\xi_{j+1}|^2} {r_j-r_{j+1}}} \Big) |e(0,\xi_\ell)|^{n-2}
e^{-c_0\frac{|\xi_n|^2}{r_n}}
e^{-c_0 \frac{r_{n-2}-r_n}{\mu(\xi_{n-1},1/\tau)^2}} \\
& \hspace{2.5in} \times|e(0,\xi_{n-1})||e(0,\xi_n)|   e^{-c_0 \frac{r_{n-1}}{\mu(\xi_{n},1/\tau)^2}}
e^{-c_0 \frac{r_n}{\mu(0,1/\tau)^2}}
\, d\xi_n\cdots d\xi_1 \\
& \leq A^n e^{-\frac{c_0}2 \frac{|z|^2}s} \sumjko |\ajkoo|^{n-2} s^{(n-2)\fjkt}(n-2)^{(n-2)\fjkt}
\int_{\C^n} \Big[ \prod_{m=1}^n e^{-\frac{c_0}2 \frac{r_{m-1}}{(r_{m-1}-r_m)r_m} |\xi_m - \frac{r_m}{r_{m-1}} \xi_{m-1}|^2}\Big] \\
&\times \Big( |e(0,\xi_n)| e^{-\frac{c_0}4 \frac{|\xi_n|^2}{r_n}} e^{-c_0 \frac{r_{n-1}}{\mu(\xi_n,1/\tau)^2}}\Big)
\Big( |e(0,\xi_{n-1})| e^{-\frac{c_0}4 \frac{|\xi_{n-1}|^2}{r_{n-1}}} e^{-c_0 \frac{r_n}{\mu(\xi_{0},1/\tau)^2}}
e^{-c_0 \frac{r_{n-2}-r_n} {\mu(\xi_{n-1},1/\tau)^2}} \Big)\, d\xi_n\cdots d\xi_1.
\end{align*}

Again using (\ref{eqn:polynomial times exponential decay}), we have
\begin{multline*}
|e(0,\xi_n)| e^{-\frac{c_0}4 \frac{|\xi_n|^2}{r_n}} e^{-c_0 \frac{r_{n-1}}{\mu(\xi_n,1/\tau)^2}}
\leq C \sum_{j\geq 1} |\A{j1}{\xi_n}| |\xi_n|^j \frac{r_n^{j/2}}{|\xi_n|^j} \frac{r_{n-1}^{1/2}}{r_{n-1}^{1/2}}
\frac{\mu(\xi_n,1/\tau)^{j+1}}{r_{n-1}^{(j+1)/2}}\\ 
= \frac{C}{r_{n-1}^{1/2}} \Lambda(\xi_n,\mu(\xi_n,1/\tau)) \leq
\frac C{\tau r_{n-1}^{1/2}}.
\end{multline*}
Since, $\max\{r_n, r_{n-2}-r_n\}\geq \frac 12 r_{n-2}$, a similar argument shows that
\[
|e(0,\xi_{n-1})| e^{-\frac{c_0}4 \frac{|\xi_{n-1}|^2}{r_{n-1}}} e^{-c_0 \frac{r_n}{\mu(0,1/\tau)^2}}
e^{-c_0 \frac{r_{n-2}-r_n} {\mu(\xi_{n-1},1/\tau)^2}}
\leq \frac C{\tau r_{n-2}^{1/2}}.
\]
Consequently,
\[
II \leq \frac{A^n e^{-c_0 \frac{|z|^2}{2s}}}{\tau^2} \sumjko |\ajkoo|^{n-2} s^{(n-2)\fjkt}(n-2)^{(n-2)\fjkt}
\Big( \prod_{m=1}^n \frac{(r_{m-1}-r_m)r_m}{r_{m-1}}\Big) r_{n-2}^{-1/2} r_{n-1}^{-1/2}.
\]
The time $r$-integrals become
\begin{multline*}
\frac 1s \int_0^s \cdots \int_0^{r_{n-1}} (r_1-r_2)^{\frac 12 -1}\cdots (r_{n-1}-r_n)^{\frac 12 -1}
r_{n-2}^{\frac 12 -1} r_{n-1}^{\frac 12 -1} r_n^{\frac 12 -1}\, dr_n\cdots dr_1 \\
 = \frac {\pi^2}s \int_0^s \cdots \int_0^{r_{n-3}} (r_1-r_2)^{\frac 12 -1}\cdots (r_{n-3}-r_{n-2})^{\frac 12 -1}
r_{n-2}^{\frac 12 -1}\, dr_{n-2}\cdots dr_1 \\
= \frac{\pi^{2+\frac{n-2}2}}s \frac{1}{\frac{n-2}2 \Gamma(\frac{n-2}2)} s^{\frac{n-2}2}. 
\end{multline*}
Thus using similar arguments to those at the end of the first estimate, we obtain
\[
|(\Mtp^{z,w})^n\Htp(s,z,w)| \leq  \frac{A^n}{s\tau^2} e^{-c_0 \frac{|z|^2}{2s}} \sjk |\ajkz|^{n-2} s^{(n-2)\fjkt} (n-2)^{(n-2)\fjkt} 
\]
which establishes (\emph{i}) with decay.

\subsection{Proof of Theorem \ref{thm:Htp satisfies QSE}, \emph{(iii)} with no decay in $\tau$.}
\label{subsec:two derivatives, no decay in tau}
Let $r_{|J|+1}=0$ and $\xi_{|J|+1} = w =0$.
The starting point for \emph{(iii)} is Proposition \ref{prop:n derivs, recursion}. We consider the case when $X^2 = \Wbstpw\Zstpz$
and outline the differences needed for other second derivative combinations later.
We have
\begin{align}
&\Wbstpw\Zstpz(\Mtp^{z,w})^n\Htp(s,z,0) \nn \\
&= n! \sum_{J\in\I_n} (-1)^{J_2} \int_0^s \int_0^{r_1} \cdots \int_0^{r_{|J|-1}} \int_{\C^{|J|}}
\Zstpz\Htp(s-r_1,z,\xi_1) \Big( \prod_{j=1}^{|J|-1} N(a_j,\xi_j) \Htp(r_j-r_{j+1},\xi_j,\xi_{j+1})\Big) \label{eqn:main term in WZMH} \\
&\hspace{2in}\times N(a_{|J|},\xi_{|J|}) \Wbstpw\Htp(r_{|J|},\xi_{|J|},0)\, d\xi_{|J|}\cdots d\xi_1\, dr_{|J|}\cdots dr_1\nn \\
&\ + n!\sum_{k=1}^{|J|} \sum_{J\in\I_n} (-1)^{J_2} \int_0^s \int_0^{r_1} \cdots \int_0^{r_{|J|-1}} \int_{\C^{|J|}} \Zstpz\Htp(s-r_1,z,\xi_1) 
\label{eqn:error term in WZMH}  \\
&\ \times  \Big( \prod_{\atopp{j=1}{j\neq k}}^{|J|} N(a_j,\xi_j) \Htp(r_j-r_{j+1},\xi_j,\xi_{j+1})\Big)
\Big(\frac{\p}{\p \w}N(a_k,\xi_k)\Big) \Htp(r_k-r_{k+1},\xi_k,\xi_{k+1}) 
\, d\xi_{|J|}\cdots d\xi_1\, dr_{|J|}\cdots dr_1\nn
\end{align}
The first integral is the most difficult to bound. We concentrate on that integral and mention at the end how to deal with integrals in the second sum.

The issue  is the convergence of the time integrals. Each spacial derivative of  $\Htp$ increases the power $s$ (or $(r_j-r_{j+1})$) 
in the denominator by $1/2$, so we have to be careful in our estimation. The trick here is to use the $e(0,\xi_n)$ term. As above, we demonstrate
the estimation on
\begin{multline*}
III := \Bigg| \int_0^s \int_0^{r_1} \cdots \int_0^{r_{n-1}} \int_{\C^{n}}\Zstpz \Htp(s-r_1,z,\xi_1)
\Big( \prod_{j=1}^{n-1} e(0,\xi_j) Z_{\tau p,\xi_j} \Htp(r_j-r_{j+1},\xi_j,\xi_{j+1})\Big) \\
e(0,\xi_n) Z_{\tau p,\xi_n} \Wbstpw\Htp(r_n,\xi_n,0)\, d\xi_n\cdots d\xi_1\, dr_n\cdots  dr_1\Bigg|.
\end{multline*}
Using (\ref{eqn:Zbstp Htp estimates}) and Lemma \ref{lem:sum to a power}, we have
\begin{multline*}
III \leq \frac{A^n}{n-1} \int_0^s\cdots\int_0^{r_{n-1}}\int_{\C^n}
\frac{e^{-c_0 \frac{|z-\xi_1|^2}{s-r_1}}}{(s-r_1)^{3/2}}
\Big(\frac{\prod_{m=1}^{n-1} e^{-c_0 \frac{|\xi_m-\xi_{m+1}|^2}{r_m-r_{m+1}}} }
{(r_m-r_{m+1})^{3/2}}\Big) 
\frac{e^{-c_0}\frac{|\xi_n|^2}{r_n}}{r_n^2}\\
|e(0,\xi_n)|\big( |e(0,\xi_1)|^{n-1} + \cdots + |e(0,\xi_{n-1})|^{n-1}\big)\, d\xi_n\cdots d\xi_1\, dr_n\cdots dr_1. 
\end{multline*}
As above, we concentrate on the space integral first. We estimate
\begin{align*}
\int_{\C^n}&e^{-c_0 \frac{|z-\xi_1|^2}{s-r_1}}  
\Big[\prod_{m=1}^{n-1} e^{-c_0 \frac{|\xi_m-\xi_{m+1}|^2}{r_m-r_{m+1}}}\Big]e^{-c_0\frac{|\xi_n|^2}{r_n}}
|e(0,\xi_n)| |e(0,\xi_\ell)|^{n-1} \, d\xi_n\cdots d\xi_1 \\
\leq C& \int_{\C^n}e^{-c_0 \frac{|z-\xi_1|^2}{s-r_1}}  
\Big[\prod_{m=1}^{n-1} e^{-c_0 \frac{|\xi_m-\xi_{m+1}|^2}{r_m-r_{m+1}}}\Big]e^{-c_0\frac{|\xi_n|^2}{r_n}}
\big(\sumjko |\ajkoo| |\xi_\ell |^{j+k} \big)^{n-1} |e(0,\xi_n)| \, d\xi_n\cdots d\xi_1 \\
\leq C^n& \int_{\C^n} e^{-\frac{c_0}4\frac{|z|^2}s}
\bigg( \prod_{m=1}^n e^{-\frac{c_0}4 \frac{r_{m-1}}{(r_{m-1}-r_m)r_m} |\xi_m - \frac{r_m}{r_{m-1}}\xi_{m-1}|^2} \bigg)\\
&\times e^{-\frac{c_0}4 \frac{|\xi_n|^2}{r_n}} \bigg( \sjk |\ajko| |\xi_n|^{j+k-1}\bigg)
\sjk \big( |\ajko| |\xi_\ell|^{j+k-1}\big)^{n-1} e^{-\frac{c_0}4 \frac{|\xi_\ell|^2}{r_\ell}}\, d\xi_n\cdots d\xi_1.
\end{align*} 
By (\ref{eqn:polynomial times exponential decay}) and the fact that $j,k\geq 1$,
\[
e^{-\frac{c_0}4 \frac{|\xi_n|^2}{r_n}} |\xi_n|^{j+k-1} \leq 
\bigg(\frac{2(j+k-1)r_n}{c_0 e} \bigg)^{\frac{j+k-1}2} \leq
\bigg( \frac{2(j+k-1)s}{c_0 e} \bigg)^{\frac{j+k-1}2} \frac{r_n^{1/2}}{s^{1/2}}
\]
(where the last inequality uses $r_n \leq s$) and
\[
e^{-\frac{c_0}4 \frac{|\xi_\ell|^2}{r_\ell}} |\xi_\ell|^{(n-1)(j+k-1)} 
\leq \bigg( \frac{2(n-1)(j+k-1)s}{c_0 e} \bigg)^{(n-1)\frac{j+k-1}2}
\]
Consequently,
\begin{align*}
\int_{\C^n}&e^{-c_0 \frac{|z-\xi_1|^2}{s-r_1}}  
\Big[\prod_{m=1}^{n-1} e^{-c_0 \frac{|\xi_m-\xi_{m+1}|^2}{r_m-r_{m+1}}}\Big]e^{-c_0\frac{|\xi_n|^2}{r_n}}
|e(0,\xi_n)| |e(0,\xi_\ell)|^{n-1} \, d\xi_n\cdots d\xi_1 \\
&\leq  A^n \frac{e^{-\frac{c_0}4 \frac{|z|^2}{s}}}{s^{1/2}} \sjk \big( |\ajko|^n s^{n\frac{j+k-1}2} n^{n\frac{j+k-1}2} \big)
\frac{(s-r_1)r_1}s \Big(\prod_{m=1}^{n-1} \frac{ (r_m-r_{m+1})r_{m+1}}{r_m}\Big)r_n^{1/2}
\end{align*}
Proceeding as before and integrating the time derivatives using (\ref{eqn:convolution of s and (r-s)}) yields the estimate (\emph{iii}) for 
$\Wbstpw\Zstpz (\Mtpzw)^n \Htp(s,z,w)$ with no decay in $\tau$. 

\subsection{Proof of Theorem \ref{thm:Htp satisfies QSE}, \emph{(iii)} with decay in $\tau$.}
\label{subsec:decay in tau, 2 derivatives}
To prove the estimates with decay in $\tau$, we estimate  the space integral first.
We use Lemma \ref{lem:sum to a power} on $\prod_{j=1}^{n-3} |\xi(0, \xi_j)|$ and thus we must estimate 
the following term.
\begin{align*}
IV&:=\int_{\C^n}e^{-c_0 \frac{|z-\xi_1|^2}{s-r_1}}  
\Big[\prod_{m=1}^{n-1} e^{-c_0 \frac{|\xi_m-\xi_{m+1}|^2}{r_m-r_{m+1}}}\Big]e^{-c_0\frac{|\xi_n|^2}{r_n}}
|e(0,\xi_\ell)|^{n-3} |e(0,\xi_{n-2}) e(0,\xi_{n-1}) e(0,\xi_n)| \, d\xi_n\cdots d\xi_1 \\
&\leq \int_{\C^{n-2}}e^{-\frac{c_0}2 \frac{|z-\xi_1|^2}{s-r_1}}  
\Big[\prod_{m=1}^{n-3} e^{-\frac{c_0}2 \frac{|\xi_m-\xi_{m+1}|^2}{r_m-r_{m+1}}}\Big]e^{-\frac{c_0}2\frac{|\xi_{n-2}|^2}{r_{n-2}}}
|e(0,\xi_\ell)|^{n-3} |e(0,\xi_{n-2})| \\
&\times  \int_{\C^2} \Big[\prod_{\alpha=n-1}^{n} e^{-\frac{c_0}2 \frac{r_{\alpha-1}}{(r_{\alpha-1}-r_\alpha)r_\alpha}
 |\xi_\alpha - \frac{r_\alpha}{r_{\alpha-1}}\xi_{\alpha-1}|^2} \Big]
|e(0,\xi_{n-1})| e^{-\frac{c_0}4 \frac{|\xi_{n-1}|^2}{r_{n-1}}} e^{-c_0 \frac{r_n}{\mu(0,1/\tau)^2}} e^{-c_0 \frac{r_{n-2}-r_n} {\mu(\xi_{n-1},1/\tau)^2}}\\
&\times |e(0,\xi_n)| e^{-\frac{c_0}4 \frac{|\xi_n|^2}{r_n}}e^{-c_0 \frac{r_{n-1}}{\mu(\xi_n,1/\tau)^2}} \, d\xi_n\cdots d\xi_1
\end{align*}
Using arguments similar to the ones used in Section \ref{subsec:al4} we can estimate
\[
|e(0, \xi_\ell)|^{n-3} e^{-c_0 \frac{|\xi_\ell|^2}{4 r_\ell}} \leq 
\sjk |\ajko|^{n-3} s^{(n-3)\frac{j+k-1}2} (n-3)^{(n-3)\frac{j+k-1}2},
\]
\[
|e(0, \xi_{n-2})| e^{-c_0 \frac{|\xi_{n-2}|^2}{4 r_{n-2}}} \leq 
\sjk |\ajko| r_{n-2}^{(\frac{j+k-1}2)} \leq \sjk |\ajko| s^{(\frac{j+k-1}2)} \frac{r_{n-2}^{1/2}}{s^{1/2}}
\]
and
\begin{align*}
|e(0,\xi_n)|& e^{-\frac{c_0}4 \frac{|\xi_n|^2}{r_n}}e^{-c_0 \frac{r_{n-1}}{\mu(\xi_n,1/\tau)^2}} 
\leq C \sum_{j\geq 1} |\A{j1}{\xi_n}| |\xi_n|^j \frac{r_n^{\frac j2} r_{n-1}^{\frac 12} }{|\xi_n|^j r_{n-1}^{\frac 12}}
\frac{\mu(\xi_n,1/\tau)^{j+1}}{r_{n-1}^{\frac{j+1}2}} \leq \frac C\tau \frac{r_n^{1/4}}{r_{n-1}^{3/4}}.
\end{align*}
Since $\max\{r_n, r_{n-2}-r_n\}\geq \frac 12 r_{n-2}$, a similar argument shows that
\[
|e(0,\xi_{n-1})| e^{-\frac{c_0}4 \frac{|\xi_{n-1}|^2}{r_{n-1}}} e^{-c_0 \frac{r_n}{\mu(0,1/\tau)^2}} e^{-c_0 \frac{r_{n-2}-r_n} {\mu(\xi_{n-1},1/\tau)^2}}
\leq \frac{C}{\tau} \frac{r_{n-1}^{1/4}}{r_{n-2}^{3/4}}.
\]
Thus, the space integral is estimated as follows (note that we are using the integral estimates from the earlier case with $n-2$ replacing
$n$):
\begin{align*}
&\int_{\C^n}e^{-c_0 \frac{|z-\xi_1|^2}{s-r_1}}  
\Big[\prod_{m=1}^{n-1} e^{-c_0 \frac{|\xi_m-\xi_{m+1}|^2}{r_m-r_{m+1}}}\Big]e^{-c_0\frac{|\xi_n|^2}{r_n}}
|e(0,\xi_\ell)|^{n-3} |e(0,\xi_{n-2}) e(0,\xi_{n-1}) e(0,\xi_n)| \, d\xi_n\cdots d\xi_1 \\
&\leq  \bigg( \prod_{m=n-1}^n \frac{(r_{m-1}-r_m)r_m}{r_{m-1}}\bigg) r_n^{1/4} r_{n-1}^{-1/2} r_{n-2}^{-3/4} \\
&\times \int_{\C^{n-2}}e^{-\frac{c_0}2 \frac{|z-\xi_1|^2}{s-r_1}}  
\Big[\prod_{m=1}^{n-3} e^{-\frac{c_0}2 \frac{|\xi_m-\xi_{m+1}|^2}{r_m-r_{m+1}}}\Big]e^{-\frac{c_0}2\frac{|\xi_{n-2}|^2}{r_{n-2}}}
|e(0,\xi_\ell)|^{n-3} |e(0,\xi_{n-2})| \, d\xi_{n-2} \cdots d \xi_1 \\
&\leq \frac{C^n}{\tau^2s^{1/2}} e^{-c_0 \frac{|z|^2}{8 s}} \sjk |\ajko|^{n-2} s^{(n-2)\frac{j+k-1}2} (n-2)^{(n-2)\frac{j+k-1}2}
\frac{(s-r_1)r_1}s \bigg(\prod_{m=1}^{n-1} \frac{(r_m-r_{m+1})r_{m+1}}{r_m} \bigg) \frac{  r_n^{1/4} }{r_{n-2}^{1/4} r_{n-1}^{1/2}}.
\end{align*}
We can handle the time ($r$)-integrals using (\ref{eqn:convolution of s and (r-s)}) and compute
\begin{align*}
\int_0^s & \cdots \int_0^{r_{n-1}} (s-r_1)^{-1/2}(r_1-r_2)^{-1/2}\cdots (r_{n-1}-r_n)^{-1/2} r_{n-2}^{-1/4} r_{n-1}^{-1/2} r_n^{-3/4}\, dr_n\cdots dr_1\\
&= \frac{\sqrt\pi \Gamma(\frac 14)}{\Gamma(\frac 34)} 
\int_0^s  \cdots \int_0^{r_{n-2}} (s-r_1)^{-1/2}(r_1-r_2)^{-1/2}\cdots (r_{n-2}-r_{n-1})^{-1/2} r_{n-2}^{-1/4} r_{n-1}^{-3/4} \, dr_{n-1}\cdots dr_1\\
&= \frac{\pi \Gamma(\frac 14)^2}{\Gamma(\frac 34)^2} 
\int_0^s  \cdots \int_0^{r_{n-3}} (s-r_1)^{-1/2}(r_1-r_2)^{-1/2}\cdots (r_{n-3}-r_{n-2})^{-1/2} r_{n-2}^{-1/2} \, dr_{n-2}\cdots dr_1\\
&= \frac{\pi^{\frac{n-2}2}}{\Gamma(\frac{n-1}2)} \frac{\pi \Gamma(\frac 14)^2}{\Gamma(\frac 34)^2}  s^{\frac{n-1}2-1}.
\end{align*}
Plugging in the space and time estimates into $III$ finishes the $\tau$-decay argument.

\subsection{End of the proof of Theorem \ref{thm:Htp satisfies QSE}}
The argument for \emph{(ii)} follows from the arguments that we have already done. For example, if $X = \Zbstpz$ or $\Zstpz$, then
the argument for \emph{(i)} can be followed line by line. If $X = \Wbstpw$ or $\Wstpw$, then the argument to prove
\emph{(iii)} can be imitated line by line.. Thus, all that remains is to prove the estimate for
$(\Mtpzw)^n\frac{\p\Htp}{\p s}(s,z,w) = -\Wstpw\Wbstpw (\Mtpzw)^n \Htp(s,z,w)$. The issue is that none of tricks that we used earlier will work because
the integral in $r_n$ will not converge. Instead, we want to integrate by parts on the $\Zstpxi$ terms. The clean way to do this is to use the 
first line in (\ref{eqn:Box M commutator}) and integrate by parts. In the term that we have been using as our demonstration estimation, the integral
(analogous to (\ref{eqn:main term in WZMH}) above) becomes
\begin{align}
&\int_0^s \int_0^{r_1} \cdots \int_0^{r_{n-1}} \int_{\C^n} \Htp(s-r_1,z,\xi_1)
\Big( \prod_{j=1}^{n-1} Z_{\tau p,\xi_j}\big[ e(0,\xi_j)\Htp(r_j-r_{j+1},\xi_j,\xi_{j+1})\big]\Big)\nn \\
&\hspace{2in}\times Z_{\tau p,\xi_n} \big[e(0,\xi_n)\Wstpw\Wbstpw \Htp(r_n,\xi_n,0)\big]\, d\xi_n\cdots d\xi_1\, dr_n\cdots  dr_1 \nn\\
&= \label{eqn: integration by parts for 2 derivs} (-1)^n\int_0^s \int_0^{r_1} \cdots \int_0^{r_{n-1}} \int_{\C^n} W_{\tau p, \xi_1} \Htp(s-r_1,z,\xi_1)
\Big( \prod_{j=1}^{n-1}  e(0,\xi_j)W_{\tau p,\xi_{j+1}}\Htp(r_j-r_{j+1},\xi_j,\xi_{j+1})\Big) \\
&\hspace{2in}\times e(0,\xi_n) \Wstpw\Wbstpw\Htp(r_n,\xi_n,0)\, d\xi_n\cdots d\xi_1\, dr_n\cdots  dr_1. \nn
\end{align}

After this integration by parts, we can proceed as with \emph{(iii)} above. To handle the terms that arise when the $w$-derivative does not
get applied to $\Htp(r_n,\xi_n,0)$, we can use a combination of integration by parts as in \eqref{eqn: integration by parts for 2 derivs} (this will be only be needed
if $X^2 = \Wstpw\Wbstpw$) and isolating the $\p N/\p w$ term similarly to how we handled $|e(0,\xi_{n-2})|$ in \S\ref{subsec:decay in tau, 2 derivatives}.
This concludes the proof of Theorem \ref{thm:Htp satisfies QSE}.

\bibliographystyle{alpha}

\begin{thebibliography}{NRSW89}

\bibitem[BR]{BoRa10}
Albert Boggess and Andrew Raich.
\newblock The ${\Box}_b$-heat equation on quadric manifolds.
\newblock {\em to appear, J.\ Geom.\ Anal.}
\newblock arXiv:0907.0148.

\bibitem[BR09]{BoRa09}
A.\ Boggess and A.\ Raich.
\newblock A simplified calculation for the fundamental solution to the {H}eat
  {E}quation on the {H}eisenberg {G}roup.
\newblock {\em Proc.\ Amer.\ Math.\ Soc.}, 137(3):937--944, 2009.

\bibitem[Chr89]{Chr89e}
M.~Christ.
\newblock Embedding compact three-dimensional {CR} manifolds of finite type in
  {${\bf C}^n$}.
\newblock {\em Ann.\ of Math. (2)}, 129(1):195--213, 1989.

\bibitem[Chr91]{Christ91}
M.\ Christ.
\newblock On the $\bar\partial$ equation in weighted ${L}^2$ norms in
  ${{\mathbb C}}^1$.
\newblock {\em J.\ Geom.\ Anal.}, 1(3):193--230, 1991.

\bibitem[GS67]{GeSh67}
I.M.\ Gel{'}fand and G.E.\ Shilov.
\newblock {\em Generalized functions. {V}ol.\ 3: \textrm{{T}heory of
  differential equations. {T}ranslated from the {R}ussian by {M}einhard {E}.\
  {M}ayer}}.
\newblock Academic Press, New York-London, 1967.

\bibitem[Has94]{Has94}
F.~Haslinger.
\newblock Szeg{\"o} kernels for certain unbounded domains in ${{\mathbb
  C}}\sp2$. {T}ravaux de la {C}onf{\'e}rence {I}nternationale d{'}{A}nalyse
  {C}omplexe et du 7e {S}{\'e}minaire {R}oumano-{F}inlandais (1993).
\newblock {\em Rev.\ Roumaine Math.\ Pures Appl.}, 39:939--950, 1994.

\bibitem[Has95]{Has95}
F.~Haslinger.
\newblock Singularities of the {S}zeg{\"o} kernel for certain weakly
  pseudoconvex domains in ${C}\sp 2$.
\newblock {\em J.\ Funct.\ Anal.}, 129:406--427, 1995.

\bibitem[Has98]{Has98}
F.\ Haslinger.
\newblock {B}ergman and {H}ardy spaces on model domains.
\newblock {\em Illinois J.\ Math.}, 42:458--469, 1998.

\bibitem[JSC86]{JeSa86}
David~S. Jerison and Antonio S{\'a}nchez-Calle.
\newblock Estimates for the heat kernel for a sum of squares of vector fields.
\newblock {\em Indiana Univ.\ Math.\ J.}, 35(4):835--854, 1986.

\bibitem[Kan89]{Kan89}
Hyeonbae Kang.
\newblock $\bar{\partial}_b$-equations on certain unbounded weakly pseudoconvex
  domains.
\newblock {\em Trans.\ Amer.\ Math.\ Soc.}, 315:389--413, 1989.

\bibitem[Nag86]{Na86}
A.\ Nagel.
\newblock Vector fields and nonisotropic metrics.
\newblock In {\em Beijing Lectures in Harmonic Analysis}, Ann.\ of Math.\
  Stud., pages 241--306. Princeton University Press, 1986.

\bibitem[NRSW89]{NaRoStWa89}
A.\ Nagel, J.-P.\ Rosay, E.M.\ Stein, and S.\ Wainger.
\newblock Estimates for the {B}ergman and {S}zeg{\"o} kernels in ${{\mathbb
  C}}^2$.
\newblock {\em Ann.\ of Math.}, 129:113--149, 1989.

\bibitem[NS01a]{NaSt01h}
A.\ Nagel and E.M.\ Stein.
\newblock The ${\Box}_b$-heat equation on pseudoconvex manifolds of finite type
  in ${{\mathbb C}}^2$.
\newblock {\em Math.\ Z.}, 238:37--88, 2001.

\bibitem[NS01b]{NaSt01f}
A.\ Nagel and E.M.\ Stein.
\newblock Differentiable control metrics and scaled bump functions.
\newblock {\em J.\ Differential Geometry}, 57:465--492, 2001.

\bibitem[NS06]{NaSt06}
A.\ Nagel and E.M.\ Stein.
\newblock The $\dbar_b$-complex on decoupled domains in ${{\mathbb C}}^n$, $n
  \geq 3$.
\newblock {\em Ann.\ of Math.}, 164:649--713, 2006.

\bibitem[NSW85]{NaStWa85}
A.\ Nagel, E.M.\ Stein, and S.\ Wainger.
\newblock Balls and metrics defined by vector fields {I}: Basic properties.
\newblock {\em Acta Math.}, 155:103--147, 1985.

\bibitem[Rai]{Rai10h}
Andrew Raich.
\newblock Heat equations and the weighted $\dbar$-problem with decoupled
  weights.
\newblock {\em submitted}.
\newblock arXiv:0704.2768.

\bibitem[Rai06a]{Rai06h}
Andrew Raich.
\newblock Heat equations in ${{\mathbb R}}\times{{\mathbb C}}$.
\newblock {\em J.~Funct.\ Anal.}, 240(1):1--35, 2006.

\bibitem[Rai06b]{Rai06f}
Andrew Raich.
\newblock One-parameter families of operators in ${\mathbb{\C}}$.
\newblock {\em J.\ Geom.\ Anal.}, 16(2):353--374, 2006.

\bibitem[Rai07]{Rai07}
Andrew Raich.
\newblock Pointwise estimates of relative fundamental solutions for heat
  equations in ${{\mathbb R}}\times{{\mathbb C}}$.
\newblock {\em Math.\ Z.}, 256:193--220, 2007.

\bibitem[Rud87]{Rud87}
Walter Rudin.
\newblock {\em Real and complex analysis}.
\newblock McGraw-{H}ill {B}ook {C}o., {T}hird edition, 1987.

\bibitem[Sik04]{Sik04}
A.\ Sikora.
\newblock Riesz transform, {G}aussian bounds and the method of the wave
  equation.
\newblock {\em Math.\ Z.}, 247:643--662, 2004.

\bibitem[Str09]{Str09h}
Brian Street.
\newblock The ${\Box}_b$-heat equation and multipliers via the wave equation.
\newblock {\em Math.\ Z.}, 263(4):861--886, 2009.

\end{thebibliography}

\end{document}